\documentclass[9pt]{amsart}
\usepackage{amssymb}
\usepackage{amsmath}
\usepackage{amsfonts} 
\usepackage{amsrefs}
\theoremstyle{plain}
\newtheorem{theorem}{Theorem}[section]

\newtheorem{definition}[theorem]{Definition}
\newtheorem{remark}[theorem]{Remark}

\newtheorem{example}[theorem]{Example}

\theoremstyle{remark}
\usepackage{xy}
\makeatletter
\newcommand{\xyC}[1]{%
\makeatletter
\xydef@\xymatrixcolsep@{#1}
\makeatother
} 
\newcommand{\xyV}[1]{%
\makeatletter
\xydef@\xymatrixrowsep@{#1}
\makeatother
} 
\input xy
\xyoption{all}
 
\begin{document}

\title{On  structure sets  of  manifold pairs} 

\author{Matija Cencelj}      
\address{Institute of Mathematics, Physics and Mechanics, 
         and Faculty of Education,
         University of Ljubljana,
         Jadranska 19,
         SI-1000 Ljubljana
         Slovenia}
\email{matija.cencelj@guest.arnes.si}

\author{Yuri V. Muranov}
\address{Vitebsk State University,
         Moskovskii pr. 33,
         210026 Vitebsk, Belarus}
\email{ymuranov@mail.ru}

\author{Du\v{s}an Repov\v{s}}
\address{Faculty of Mathematics and Physics, 
         and Faculty of Education,
         University of Ljubljana,
         Jadranska 19,
         SI-1000 Ljubljana
         Slovenia}
\email{dusan.repovs@guest.arnes.si} 

 
 
\begin{abstract}
In this paper we systematically describe relations between various structure sets
which arise naturally for  pairs of compact topological manifolds with boundary.
Our consideration is based on a deep analogy between the case of a compact 
manifold with boundary and the case of a
closed manifold pair. This approach also gives
a possibility to construct the obstruction groups for natural maps of various 
structure sets and to investigate their properties.
\end{abstract}

\maketitle 

\section{Introduction}

Let
\begin{equation}\label{1.1}
Q=\left[
\xy\drop!C\xybox{\xygraph{%
!M @=1.6pc{
 X^n & \partial X^n \ar[l]   \\
 Y^{n-q} \ar[u] & \partial Y^{n-q} \ar[l] \ar[u]	
 } 
}}
\endxy
\right]
\end{equation}
be a pair of compact topological  manifolds with boundary
such that $\partial Y\subset \partial X$ and 
$Y\subset X$ are  locally flat submanifolds with given
structures of  the 
normal bundles (cf. \cite{Ranicki79} and  \cite{Ranicki81}).

The pair of manifolds in (\ref{1.1}) is a special case of a stratified
manifold (cf. \cite{Ranicki81} and \cite{Weinberger}) and a complete
description of the
relations between various structure sets in this situation is
very helpful for investigation of the general situation of a stratified
manifold. In this paper we systematically describe relations between
different structure sets which arise naturally in this case. We
shall work in the category of topological manifolds and assume
that the
dimension of all manifolds is
greater than or equal to 5.  We shall
use surgery theory for the case of simple homotopy equivalence
(\cite{Ranicki79}, \cite{Ranicki81}, and \cite{Wall}), and we 
shall
tacitly assume that all
obstruction groups are  decorated with "s".            
 We shall use the
 functoriality of the basic maps in surgery theory
and the algebraic surgery theory of Ranicki (cf. \cite{BakMur03}, \cite{CavMurSpa05}, \cite{HamRanTay}, \cite{Ranicki79}, and \cite{Ranicki81}).

In Section 2 we define classical structure sets for  a  manifold with
boundary $(X, \partial X)$ and we
describe the natural relations among them.
In Section 3 we define structure sets for  a  closed   manifold pair $(X^n, Y^{n-q})$
and we
describe the natural relations among them. We point out a deep analogy between
the case $(X, \partial X)$ considered in Section 2 and  the case of a closed manifold  pair $(X, Y)$ considered
 in Section 3.

 In Section 4  we present
 the necessary material on the
 realization of the structure
sets introduced
 above,
 on the spectrum level (\cite{BakMur03}, \cite{BakMur06}, \cite{CavMurSpa06}, \cite{CenMurRep}, 
 \cite{HamRanTay}, \cite{Ranicki79},  and \cite{Ranicki81}). From this realization and the results of
the previous sections,  we derive
 the well-known  basic diagrams of exact sequences  for the structure sets
(cf.   \cite{BakMur06}, \cite{CavMurSpa06}, \cite{CenMurRep}, \cite{Ranicki79}, \cite{Ranicki81}, 
and \cite{Wall}). In this section we also describe
 the  analogy between
 $(X, \partial X)$ and a closed manifold  pair $(X, Y)$  on the
level of  the natural maps between  the structure sets.
 This approach gives a possibility to better understand  the geometric properties of  natural maps
 in surgery theory.

 In Section 5 we extend our consideration to the case of a compact manifold
pair with boundary. The results of this section are known to the experts,
but not all of them have been published until now.
Our approach gives an opportunity to systematically
describe  the
structure sets of such pairs and natural maps between them. The naturality of
this approach also provides realization of structure sets and  the natural maps
between them  on the spectrum level. In this way it is possible
to describe all  structure sets for a closed manifold pair with
 boundary
and describe  relations among
them,  other structure sets,  and the surgery obstruction groups.

In Section 6 we point out some possible applications of our approach
to the description of  various obstruction groups  which arise naturally
when we consider the natural maps between structure sets.

\section{The  structure sets of $(X, \partial X)$}

Let $X^n$ be a closed $n$-dimensional   manifold. An
{\sl $s$-triangulation of $X$} is a
simple homotopy
equivalence
$f\colon M\to X$,
where $M^n$ is a closed $n$-dimensional topological manifold (cf. \cite{Ranicki79}, \cite{Ranicki81},
and \cite{Wall}).
Two $s$-triangulations
$f_i\colon M_i\to X$
are {\sl concordant} (cf. \cite{Ranicki81}) if there exists a
simple homotopy equivalence of triads
\[
(F;f_0, f_1): (W; M_0, M_1)\to (X\times I; X\times\{0\},
X\times\{1\}),
\]
where $W$ is a compact $(n+1)$-dimensional manifold
with the boundary $\partial W=M_0\cup M_1$.
The set of concordance classes  is
denoted by $\mathcal{S}(X)$.
\smallskip

A $t$-triangulation of $X$ is a topological normal map
$(f,b): M\to X$,
where $M^n$ is a closed $n$-dimensional topological manifold (cf.  \cite{Ranicki79} and \cite{Ranicki81}).
Two  $t$-triangulations  are  {\sl concordant} (cf. \cite{Ranicki81}) if there exists 
a topological normal
map of triads
\[
((F, C);(f_0,b_0), (f_1,b_1))\colon (W; M_0, M_1)\to (X\times I;
X\times\{0\}, X\times\{1\}).
\]
The set of concordance classes of $t$-triangulations of $X$ is
denoted by $\mathcal{T}(X)$.
\smallskip

For $n\geq 5$ these sets  fit into the surgery exact sequence (cf.
 \cite{Ranicki79}, \cite{Ranicki81}, and \cite{Wall})
\begin{equation}\label{2.1}
\cdots\to
 L_{n+1}(\pi_1(X))\to \mathcal{S}(X)\to \mathcal{T}(X)\to
L_n(\pi_1(X))\to\cdots
\end{equation}
where $L_*(\pi_1(X))$ are the surgery obstruction groups.

Thus for a closed manifold $X$ we have two types of structures,
which are  denoted by $\mathcal{S}$ and $\mathcal{T}$. We can write down the
following diagram for  these structures
\begin{equation}\label{2.2}
\mathcal{S}\longrightarrow \mathcal{T}. 
\end{equation}
We can interpret the arrow in (\ref{2.2}) as a relation of partial order and can say
that the structure $\mathcal{S}$ is `stronger' than the structure $\mathcal{T}$.

The diagram (\ref{2.2})  means that,  for any manifold $X$,  a  simple
homotopy equivalence of $n$-manifolds $ f\colon M\to X, $ which represents
an element of the structure set $\mathcal{S}(X)$, correctly defines  a
structure in $\mathcal{T}(X)$ (cf. \cite{Wall}), that is  the map $\mathcal{S}(X)\to 
\mathcal{T}(X)$ in (\ref{2.1}). The map of structures in (\ref{2.2}) is the
map of `weakening of the structure'.
\smallskip

Let $(X^n, \partial X^n)$ be a compact manifold with boundary.
In this case  there are  the following five structures on
$(X,\partial X)$:
\begin{equation}\label{2.3}
\mathcal{ T}\mathcal{ T}, \quad \mathcal{ T}\mathcal{ S}, \quad \mathcal{ T}\mathcal{ D}, 
\quad \mathcal{ S}\mathcal{ S} \ \
\hbox{and} \quad \mathcal{ S}\mathcal{ D}. \quad
\end{equation}
The two first  structures on   $(X,\partial X)$ in (\ref{2.3})  correspond to
 $\mathcal{ T}$-structures on $X$ whose restrictions to the boundary give the structures
$\mathcal{ T}$ and  $\mathcal{ S}$. The third structure corresponds to  $\mathcal{ T} $-structures
on $X$ whose restriction to $\partial X$ is a homeomorphism and the equivalence relation is considered 
rel 
the boundary. 
The two second structures   on $(X,\partial X)$ in (\ref{2.3})  are considered similarly.

For a  compact manifold with boundary  $(X^n, \partial X)$,  we recall the definition of 
the structure sets which correspond to the structures from (\ref{2.3})
(cf. \cite{CavMurSpa06},  \cite{Ranicki81}, and \cite{Wall}).
\smallskip

 ($\mathcal{ T} \mathcal{ T}$) This is the
 relative case in the sense
of   \cite{Ranicki81} and \cite{Wall}.
Let
\[
((f,b),(\partial f, \partial b))\colon
(M, \partial M)\to (X, \partial X)
\]
 be a  $t$-triangulation of the pair $(X, \partial X)$.
Two $t$-triangulations
\[
((f_i,b_i),(\partial f_i, \partial b_i))\colon
(M_i, \partial M_i)\to (X, \partial X), \ i=0,1
\]
are {\sl  concordant}
if  there exists a $t$-triangulation of the  $4$-ad
\begin{equation}\label{2.4}
((F,B);(g,c),(f_0, b_0), (f_1, b_1))\colon
(W; V, M_0, M_1)\to
(X\times I; \partial X\times I,  X\times\{0\},
 X\times\{1\})
\end{equation}
with
$\partial V=\partial M_0\cup \partial M_1$.
The set of concordance classes is denoted
by $\mathcal{ T}(X,\partial X)$.
\smallskip

($\mathcal{ T} \mathcal{ S}$) This  is the
 mixed structure on a manifold with boundary in
the sense of   \cite{CavMurSpa06} and \cite{Wall}.
Let  $f\colon (M, \partial M) \to (X, \partial X)$ be a $t$-triangulation  of $(X, \partial X)$
 such that  the restriction
$f|_{\partial M} \colon\partial M \to \partial X$
is an $s$-triangulation.
Two such maps  $f_i \colon (M_i, \partial M_i) \to (X, \partial X)$,
$i = 0, 1$,  are {\sl  concordant} if  there exists
a $t$-triangulation (\ref{2.4}) of the $4-ad$    such that
$V = F^{-1}(\partial X \times I)$ and
the restriction $F|_V$  is  an $s$-triangulation
\[
(F|_V; F|_{\partial M_0},  F|_{\partial M_1})\colon (V;\partial M_0, \partial M_1)\to
(\partial X\times I;  \partial X\times\{0\},
\partial  X\times\{1\})
\]
of the triad $(\partial X\times I;  \partial X\times\{0\},
\partial  X\times\{1\})$.
The set of equivalence classes of such maps
is denoted by $\mathcal{ T} \mathcal{ S}(X, \partial X)$  (cf. \cite{CavMurSpa06}).
\smallskip

($\mathcal{ T} \mathcal{ D}$)  This is
the $rel_{\partial}$-case in the sense of  \cite{Ranicki81} and  \cite{Wall} 
(cf. also \cite{CavMurSpa06} and \cite{CenMurRep}).
A $t_{\partial}$-triangulation of $(X, \partial X)$
is a $t$-triangulation
\[
((f,b),(\partial f, \partial b))\colon
(M, \partial M)\to (X, \partial X)
\]
whose restriction $\partial f$ to the boundary
is a homeomorphism $\partial M\to \partial X$.
Two $t_{\partial}$-triangulations
\[
((f_i,b_i),(\partial f_i, \partial b_i))\colon
(M_i, \partial M_i)\to (X, \partial X), \ i=0,1
\]
are  {\sl concordant} if there
 exists a $t$-triangulation (\ref{2.4}) of the $4$-ad
 with the condition
\begin{equation}\label{2.5}
 V=\partial M_0 \times I, \ \  (g,c)=\partial f_0 \times I\colon V\to \partial X\times I.
\end{equation}
The set of concordance classes is denoted
by $\mathcal{ T}^{\partial}(X,\partial X)$.
\smallskip

($\mathcal{ S} \mathcal{ S}$)  This is the
 relative case of $s$-triangulations in the sense
of   \cite{Ranicki81} and \cite{Wall}.
Let
\[
(f,\partial f)\colon (M, \partial M)\to (X, \partial X)
\]
be an  $s$-triangulation of the pair  $(X, \partial X)$.
Two $s$-triangulations
\[
(f_i,\partial f_i)\colon
(M_i, \partial M_i)\to (X, \partial X), \ i=0,1
\]
are {\sl concordant}
if there exists a simple homotopy equivalence of $4$-ads
\begin{equation}\label{2.6}
((F;g, f_0, f_1))\colon (W; V, M_0, M_1)\to (X\times I; \partial X\times
I,  X\times\{0\},
 X\times\{1\})
\end{equation}
with
\begin{equation}\label{2.7}
\partial V=\partial M_0\cup \partial M_1.
\end{equation}
The set of concordance classes is denoted
by $\mathcal{ S}(X,\partial X)$.
\smallskip

($\mathcal{ S} \mathcal{ D}$)  This is the case
of $s_{\partial}$-triangulation in the sense of
 \cite{CavMurSpa06},  \cite{CenMurRep},  and \cite{Ranicki81}.
An $s_{\partial}$-triangulation of $(X, \partial X)$
is an $s$-triangulation  of the pair $(X, \partial X)$
\[
(f,\partial f)\colon (M, \partial M)\to (X, \partial X)
\]
for which the restriction $\partial f$ is a homeomorphism.
Two $s_{\partial}$-triangulations
\[
(f_i,\partial f_i)\colon
(M_i, \partial M_i)\to (X, \partial X), \ i=0,1
\]
are {\sl concordant}
if  there exists an $s$-triangulation  (\ref{2.6})  of the  $4$-ad
and
\[
 V=\partial M_0 \times I, \ \  g=\partial f_0 \times I\colon V\to \partial X\times I.
\]
The set of concordance classes is denoted
by $\mathcal{ S}^{\partial}(X,\partial X)$.
\smallskip

For a compact manifold with boundary $X\hookleftarrow \partial X$
we can write down  the following diagram of structures
\begin{eqnarray}\label{2.8}
\mathcal{S}\mathcal{D}\longrightarrow&\mathcal{T}\mathcal{D}&\stackrel{=}{\longrightarrow}\mathcal{T}\mathcal{D}\nonumber\\
\downarrow\qquad&\downarrow&\qquad\downarrow\\
\mathcal{S}\mathcal{S}\longrightarrow&\mathcal{T}\mathcal{S}&\longrightarrow\mathcal{T}\mathcal{T}\nonumber
\end{eqnarray}
in which, similarly to the diagram (\ref{2.2}), the arrows provide natural maps of
`weakening of the structure' (cf. \cite{CavMurSpa06},  \cite{Ranicki81}, and \cite{Wall}),
 and we can interpret
every arrow
(or their composition) as a relation of a
partial order.

For any  structure in (\ref{2.8}),  the right symbol which corresponds to the
structure on
the  boundary is `stronger than or equal' to  the left symbol.
Every arrow    in (\ref{2.8}) gives a natural  map of `weakening of the structure' on a manifold
with boundary $(X, \partial X)$.
From this it follows that any sequence
 of arrows from
(\ref{2.8}) also gives a map of `weakening of the structure'. Hence,  for any $(X,\partial X)$,  
the diagram (\ref{2.8}) provides the following commutative diagram
of structure sets (compare with  the  diagram on page 116 of \cite{Wall})
\begin{eqnarray}\label{2.9}
\mathcal{S}^{\partial}(X, \partial X) \longrightarrow& \mathcal{ T}^{\partial}(X, \partial X)&
\stackrel{=}{\longrightarrow}\mathcal{ T}^{\partial}(X, \partial X)\nonumber\\
\downarrow\qquad&\downarrow&\qquad\downarrow\\
           \mathcal{ S}(X, \partial X)\longrightarrow&\mathcal{ T}\mathcal{ S} (X, \partial X)&\longrightarrow
 \mathcal{ T} (X, \partial X).\nonumber
\end{eqnarray}
\smallskip

The diagram (\ref{2.9}) gives an opportunity to construct several commutative
diagrams which have the form of the square or of the triangle, for example
we have
\begin{eqnarray}\label{2.10}
\mathcal{ S}^{\partial}(X, \partial X) &\longrightarrow&\mathcal{ T}^{\partial}(X, \partial X)\nonumber\\
\downarrow\qquad& &\qquad\downarrow\\
           \mathcal{ S}(X, \partial X) &\longrightarrow&
 \mathcal{ T} (X, \partial X).\nonumber
\end{eqnarray}
\smallskip

Excluding the trivial  cases when the restriction of the structure to the
boundary is a homeomorphism, we obtain the following relations
between structures on $(X, \partial X)$ and structures on $\partial X$
(cf. \cite{CavMurSpa06}, \cite{CenMurRep}, \cite{Ranicki81}, and \cite{Wall} )
\begin{equation}\label{2.11}
\mathcal{ S}\mathcal{ S}\to \mathcal{ S}, \quad \mathcal{ T}\mathcal{ S}\to \mathcal{ S},\quad
\mathcal{ T}\mathcal{ T}\to \mathcal{ T}.
\end{equation}

In (\ref{2.11}) the left pairs of symbols correspond to the structures on $(X, \partial X)$ and the right 
symbols correspond to the  structures on $\partial X$ which is the restriction of the structure on 
$(X, \partial X)$ to the boundary.
The   relations in (\ref{2.11})
and the diagram (\ref{2.8})  provide the following diagram of structures 
(cf. \cite{CavMurSpa06},  \cite{Ranicki81}, and \cite{Wall})

\begin{eqnarray}\label{2.12}
 \mathcal{ S}\mathcal{ S}\longrightarrow&\mathcal{ T}\mathcal{ S}&
 \longrightarrow \mathcal{ T}\mathcal{ T}\nonumber\\
\downarrow\qquad&\downarrow&\qquad\downarrow\\
\mathcal{ S}\longrightarrow&\mathcal{ S} &
 \longrightarrow \mathcal{ T}.\nonumber
\end{eqnarray}

As above, for any manifold
with boundary $(X, \partial X)$,  the diagram (\ref{2.12}) provides the following commutative diagram
of structure sets (compare with the  diagram on page 116 of \cite{Wall})

\begin{eqnarray}\label{2.13}
\mathcal{ S} (X, \partial X)\longrightarrow&\mathcal{ T}\mathcal{ S} (X, \partial X)
& \longrightarrow \mathcal{ T} (X, \partial X)\nonumber\\
\downarrow\qquad&\downarrow&\qquad\downarrow\\
\mathcal{ S} (\partial X)\stackrel{=}{\longrightarrow}&\mathcal{ S} (\partial X)&
\longrightarrow\mathcal{ T} (\partial X).\nonumber
\end{eqnarray}

\section{The  structure sets of a closed manifold pair $(X, Y)$}

Let $(X^n, Y^{n-q})$ be a closed manifold pair without boundary. In this section we systematically describe
various structures  on this pair. It follows from our description that, for
a pair of closed manifolds,  structures  and relations between
them     are similar to structures  and relations for a manifold with boundary.
 In fact
it is clear from our consideration, that in a number of cases the submanifold `plays the role of a boundary'
\cite{BakMur06}.
For a pair $(X^n\hookleftarrow Y^{n-q})$,  the list of structures almost  coincides with the corresponding list
of structures for $(X, \partial X)$. Below we explain the difference. First we remark that all structures from
(\ref{2.3}) are realized for manifold pairs
 $(X, Y)$.  Now we define  these structures (cf. \cite{BakMur06}, \cite{Ranicki81}, and \cite{Wall}).
\smallskip

($\mathcal T \mathcal T$) For a closed manifold pair $(X, Y)$ this
structure  is given by the structure set
 $\mathcal T(X, Y) = \mathcal T(X)$ (cf. \cite{Ranicki81})
 and it 
 consists of concordance classes of  $t$-triangulations of the manifold $X$.
\smallskip

($\mathcal T \mathcal S$) This structure is given by the  structure set $\mathcal N\mathcal S(X, Y)$ introduced in \cite{BakMur06}.
Its restriction  to $X$     gives the  $\mathcal T$ structure on
$X$ and its  restriction to $Y$  gives the $\mathcal S$-structure \cite{BakMur06}.
This structure is similar
to the  mixed structure on the manifold with boundary 
(cf.  \cite{BakMur06}, \cite{CavMurSpa06},  and \cite{Wall}).

Let  $f\colon M \to X$ be a $t$-triangulation which is transversal to the submanifold
$Y$ with  $N=f^{-1}(Y)$, and  the restriction
$f|_N\colon N\to Y$
is a simple homotopy equivalence.
Two such maps
\[
f_i \colon M_i\to X,  \ \ N_i= f_i^{-1}(Y),  \ \
(i=1,2)
\]
are {\sl  concordant}  if there
exists a $t$-triangulation
\[
(F;f_0, f_1)\colon
(W; M_0, M_1)\to
(X\times I;   X\times\{0\},
 X\times\{1\})
\]
 with  the following
properties:

\noindent (i) $\partial W =M_1\cup M_2$
and  $F|_{M_i}=f_i \ (i=0,1)$;

\noindent (ii) $F$ is transversal to $Y$  with $F^{-1}(Y)= V$ and $\partial V
=N_0\cup N_1$;

\noindent (iii) the restriction $F|_V$  is  an  $s$-triangulation of the triad
\[
\left(F|_V; f_0|_{N_0},f_1|_{N_1}\right)  \colon (V;N_0, N_1)\to (Y\times I; Y\times\{0\},  Y\times\{1\}).
\]
The set of equivalence classes of such maps
is denoted by $\mathcal N \mathcal S(X, \partial X)$ (cf.\cite{BakMur06}).
\smallskip

($\mathcal T \mathcal D$) The structure set in this case is given
by the  $rel_{\partial}$ $t$-triangulations of  the manifold with boundary $(X\setminus Y, \partial(X\setminus Y))$, that
is
$\mathcal T^{\partial}(X\setminus Y, \partial(X\setminus Y))$
(cf.  \cite{CavMurSpa06}, \cite{CenMurRep},  \cite{Ranicki81}, and \cite{Wall}).
\smallskip

($\mathcal S \mathcal S$) This structure is given by the structure set $\mathcal S(X, Y, \xi)$
where $\xi$ is a topological normal block bundle of $Y$ in $X$ 
(cf. \cite{Ranicki81}).
The restriction of this structure to $X$ is an   $\mathcal S$-structure, and the
restriction to $Y$ also is an  $\mathcal S$-structure.
We  give now an explicit  definition following \cite{Ranicki81}.

An $s$-triangulation of a manifold pair $\mathcal S(X, Y, \xi)$ is a  $t$-triangulation
\[
f\colon M\to X,  \ \ g=f|_{N}\colon N\to Y,
\]
which is transversal to $Y$ with $N=f^{-1}(Y)$ and
for which the maps
\[
g=f_N\colon N\to Y,  \ \  \text{and} \ \ f|_{M\setminus N}=h \colon
 (M\setminus N, \partial (M\setminus N))\to (X\setminus Y, \partial (X\setminus Y))
 \]
 are $s$-triangulations (cf. \cite{BakMur06}, \cite{CenMurRep}, and \cite{Ranicki81}).
Two such  $s$-triangulations
\[
(f_; g_i, h_i)\colon
(M_i; N_i, M_i\setminus N_i)\to (X; Y, X\setminus Y), \ i=0,1
\]
are {\sl concordant}
if  there exists an $s$-triangulation
\[
(F;G, H)\colon
(W; V, W\setminus V)\to (X\times I; Y\times I,  (X\setminus Y)\times I)
\]
with
\[
\partial W= M_0\cup  M_1, \ \ \partial V= N_0\cup  N_1,
\]
\[
\partial(W\setminus V)=  (M_0\setminus N_0)\cup
 (M_1\setminus N_1)\cup F^{-1}\left(\partial(X\setminus Y)\times I\right),
\]
\[
F_{M_i}=f_i, \ \
G_{M_i}= g_i, \ \
H_{M_i\setminus N_i} = h_i.
\]
The set of concordance classes is denoted
by $\mathcal S(X, Y, \xi)$.
\smallskip

($\mathcal S \mathcal D$) In this case we shall consider $\mathcal S$-structures on $X$  whose restriction
to a tubular neighborhood of $Y$ provide  $\mathcal D$-structures.  The structure set in this case is given
by the  $rel_{\partial}$  $s$-triangulations  of  the manifold with boundary $(X\setminus Y, \partial(X\setminus Y))$
that is
\[
\mathcal S^{\partial}(X\setminus Y, \partial(X\setminus Y))
\]
(cf. \cite{CavMurSpa06}, \cite{CenMurRep},  \cite{Ranicki81}, and \cite{Wall}).
\smallskip

The relations  (\ref{2.8}) take place for the structures introduced above
in this section
(cf. \cite{BakMur06}, \cite{CavMurSpa06}, \cite{CenMurRep}, \cite{Ranicki81},  and \cite{Wall}).
 Thus, for a closed  manifold pair
$(X, Y)$,  the diagram (\ref{2.8}) provides the following commutative diagram
of structure sets
\begin{eqnarray}\label{3.1}
\mathcal S^{\partial}(X\setminus Y, \partial (X\setminus Y)) \longrightarrow&
\mathcal T^{\partial}(X\setminus Y, \partial (X\setminus Y))& \stackrel{=}{\longrightarrow}
\mathcal T^{\partial}(X\setminus Y, \partial (X\setminus Y))\nonumber\\
\downarrow\qquad&\downarrow&\qquad\downarrow\\
\mathcal S(X, Y, \xi ) \longrightarrow&\mathcal N\mathcal S (X,  Y) &\longrightarrow \mathcal T (X)\nonumber
\end{eqnarray}
which is similar to diagram (\ref{2.9}).
In a similar way (cf. \cite{BakMur06}, \cite{CavMurSpa06}, \cite{CenMurRep}, and \cite{Ranicki81}), 
using the restriction of a structure
on the manifold $X$ to the submanifold $Y$,   we obtain the  relations between
 structures on $(X,Y)$ and structures on $Y$ given in  (\ref{2.11}).
Now in (\ref{2.11}) the left pairs of symbols for any arrow  correspond to structures on $(X, Y)$ and the right
symbols correspond to the  structures on $Y$. The results of
\cite{BakMur06}, \cite{Ranicki81}, and \cite{Wall}  provide the  relations (\ref{2.12}) between
structures for a closed manifold pair $(X, Y)$.
Thus, for any closed manifold
pair  $(X, Y)$,  the diagram (\ref{2.12}) provides the following commutative diagram
of structure sets that is similar to (\ref{2.13})
\begin{eqnarray}\label{3.2}
\mathcal S (X, Y, \xi) \longrightarrow& \mathcal N\mathcal S (X, Y)  &\longrightarrow  \mathcal T (X)\nonumber\\
\downarrow\qquad&\downarrow&\qquad\downarrow\\
\mathcal S (Y)\stackrel{=}{\longrightarrow}&\mathcal S (Y)&\longrightarrow\mathcal T (Y).\nonumber
\end{eqnarray}

However, there is a  difference between structures for the  pairs  $(X, \partial X)$ and $(X, Y)$.
For a manifold pair $(X,Y)$ we can consider also the $\mathcal S$-structure on $X$ as $\mathcal S\mathcal T$-structure on
the pair $(X,Y)$. There is no analogue structure for the case of  manifolds with boundary.
In this case the restriction of the structure to the submanifold gives a `weaker
structure' than the structure on the ambient manifold.  
This structure fits into the following commutative diagram of
structures (cf. \cite{BakMur06} and \cite{Ranicki81})
\begin{eqnarray}\label{3.3}
\mathcal S\mathcal D&\stackrel{=}{\longrightarrow}& \mathcal S\mathcal D \nonumber\\
\downarrow&\qquad&\downarrow\nonumber\\
\mathcal S\mathcal S&\longrightarrow& \mathcal S\mathcal T \\
\downarrow&\qquad&\downarrow\nonumber\\
\mathcal T\mathcal S&\longrightarrow& \mathcal T\mathcal T \nonumber
\end{eqnarray}
which has no analog for the structures on manifolds with boundary.
For any closed manifold pair $(X,Y)$ the diagram (\ref{3.3}) provides the following
commutative diagram of structure sets
\begin{eqnarray}\label{3.4}
\mathcal S^{\partial} (X\setminus Y, \partial(X\setminus Y))&\stackrel{=}{\longrightarrow}&
 \mathcal S^{\partial} (X\setminus Y, \partial(X\setminus Y))\nonumber\\
\downarrow&\qquad&\downarrow\nonumber\\
\mathcal S(X, Y, \xi)&\longrightarrow&\mathcal S(X) \\
\downarrow&\qquad&\downarrow\nonumber\\
\mathcal S\mathcal N(X, Y)&\longrightarrow&\mathcal T(X). \nonumber
\end{eqnarray}

The structure  $\mathcal S\mathcal T$ fits also into the following diagram of structures
\begin{eqnarray}\label{3.5}
\mathcal S\mathcal S\longrightarrow&\mathcal S\mathcal T&\longrightarrow\mathcal T\mathcal T\nonumber \\
\downarrow\qquad&\downarrow&\qquad\downarrow\\
\mathcal S\longrightarrow& \mathcal T&\stackrel{=}{\longrightarrow} \mathcal T \nonumber
\end{eqnarray}
in which the vertical arrows correspond to the restriction of the corresponding structure
to the submanifold. For any closed manifold pair $(X,Y)$, the diagram (\ref{3.5}) gives the following
commutative diagram of structure sets
\begin{eqnarray}\label{3.6}
\mathcal S(X,Y,\xi)\longrightarrow&\mathcal S(X)&\longrightarrow\mathcal T(X)\nonumber \\
\downarrow\qquad&\downarrow&\qquad\downarrow\\
\mathcal S(Y)\longrightarrow&\mathcal T(Y)& \stackrel{=}{\longrightarrow}\mathcal T(Y). \nonumber
\end{eqnarray}

\section{The spectrum level. The diagrams  of  structure sets}

In accordance with the algebraic surgery theory of Ranicki (cf. \cite{BakMur03}, \cite{CavMurSpa06}, \cite{HamRanTay}, and \cite{Ranicki79})
the  structure sets introduced above,  various surgery obstruction groups, and natural maps
are realized on the spectrum level. We shall use the following notations.

Let $\mathbf L_{\bullet}$ be   the 1-connected cover of the
simply connected surgery $\Omega$--spectrum
$\mathbf L_{\bullet}(1)$ with
$\mathbf {L_{\bullet}}_0 \simeq G/TOP$ and $\pi_n(\mathbf L_{\bullet}(1))=L_n(1)  \ (n>0)$ (cf. \cite{Ranicki79} and  \cite{Ranicki81}).
The  cofibration
\begin{equation}\label{4.1}
X_+\land \mathbf L_{\bullet}\rightarrow \mathbb  L(\pi_1(X))\to \mathbb S(X).
\end{equation}
is defined for any   topological space  $X$ (cf. \cite{Ranicki79} and \cite{Ranicki81}),
and the algebraic surgery  exact sequence
\begin{equation}\label{4.2}
\cdots\rightarrow L_{n+1}(\pi_1(X))\rightarrow \mathcal  S_{n+1}(X)
\rightarrow
H_n(X; \mathbf L_{\bullet})\rightarrow L_{n}(\pi_1(X))\rightarrow
\cdots
\end{equation}
is a homotopy long exact sequence of this cofibration
with
\begin{equation}\label{4.3}
\pi_{j}(\mathbb S(X)) =\mathcal S_{j}(X), \quad  \mathcal T_j(X)= H_j(X; \mathbf L_{\bullet}).
\end{equation}
The groups $L_*(\pi_1(X))$ are the surgery obstruction groups of
Wall  \cite{Wall}. For a closed manifold $X^n$ with $n\geq 5$, we have
\begin{equation}\label{4.4}
\mathcal S(X) \cong \mathcal S_{n+1}(X), \quad \mathcal T(X)\cong  \mathcal T_n(X)= H_n(X; \mathbf L_{\bullet}).
\end{equation}
 For the structure sets of a compact manifold   $(X^n, \partial X^n)$
with boundary  $\partial X$ there
 exists the following homotopy commutative diagram of spectra
 (cf. \cite{CavMurSpa06}, \cite{CenMurRep}, \cite{Ranicki79}, and   \cite{Ranicki81})
\begin{equation}\label{4.5}
\xymatrix{
 & \vdots \ar[d] & \vdots \ar[d] & \vdots \ar[d]   & &\\
\dots \ar[r] & \Omega\mathbb S^{\partial}(X, \partial X) \ar[r] \ar[d] & X_+\land \mathbf L_{\bullet} \ar[r] \ar[d] & \mathbb L(\pi_1( X)) \ar[r] \ar[d] & \dots \\
\dots \ar[r] & \Omega \mathbb S(X, \partial X) \ar[r] \ar[d] & (X/\partial X)_+\land \mathbf L_{\bullet} \ar[r] \ar[d] & \mathbb L^{rel} \ar[r] \ar[d]& \dots\\
\dots \ar[r] & \mathbb S(\partial X) \ar[r] \ar[d] & \Sigma (\partial X)_+\land \mathbf L_{\bullet} \ar[d] \ar[r] &\Sigma \mathbb L(\pi_1(\partial X)) \ar[d] \ar[r] &\dots \\
  & \vdots  & \vdots  &\vdots  & }
 \end{equation}
where $\mathbb L^{rel}=\mathbb L (\pi_1(\partial X)\to \pi_1(X))$ is a spectrum for the relative $L$-groups.
The homotopy long exact sequences of the maps from (25) give
the commutative diagram of exact sequences
(cf.  \cite{Ranicki81} and \cite{Wall})
\begin{equation}\label{4.6}
\xymatrix{
\ar[d]&     &\ar[d] &    & \ar[d]        \\
\to \mathcal S^{\partial}_{n+1}(X,\partial X)
&\to&
H_n(X; \mathbf L_{\bullet})& \to & L_n(\pi_1(X))\to \\
\ar[d]& \ast   &\ar[d] &    & \ar[d]    \\
\to
\mathcal S_{n+1}(X,\partial X) &\to &H_n(X,\partial X;\mathbf L_{\bullet})&
\to &L_n^{rel}\to \\
\ar[d]& \ast\ast   &\ar[d] &    & \ar[d]   \\
\to \mathcal S_{n}(\partial X)\ar[d]
&\to&
H_{n-1}(\partial X; \mathbf L_{\bullet})\ar[d]& \to &
L_{n-1}(\pi_1(\partial X))\ar[d]\to\\
&&&&}
\end{equation}
and  we have the following isomorphisms
of the structure sets
\begin{equation}\label{4.7}
\xymatrix{
\mathcal T^{\partial}(X, \partial X) \cong H_n(X; \mathbf L_{\bullet}),& \quad&
\mathcal T(X, \partial X)\cong H_*(X, \partial X; \mathbf L_{\bullet}),\\
\mathcal T (\partial X)\cong H_{n-1}(\partial X; \mathbf L_{\bullet}), & \quad &
\mathcal S^{\partial}(X, \partial X) \cong  \mathcal S^{\partial}_{n+1}(X, \partial X), \\
\mathcal S(X, \partial X)\cong \mathcal S_{n+1}(X, \partial X), &\quad &
\mathcal S (\partial X)\cong  \mathcal S_n(\partial X)).}
\end{equation}
Note that the commutative square $*$ in (\ref{4.6}) is isomorphic
to the commutative square
\begin{eqnarray}\label{4.8}
          \mathcal S^{\partial} (X, \partial X)&\longrightarrow&\mathcal T^{\partial}
 (X, \partial X)\nonumber\\
\downarrow&&\downarrow  \\
\mathcal S (X,  \partial X)&\longrightarrow&\mathcal T (X, \partial X).\nonumber
\end{eqnarray}
which follows from the diagram  (\ref{2.9}),
and the commutative square $**$  is isomorphic  to the commutative square
\begin{eqnarray}\label{4.9}
          \mathcal S (X, \partial X) &\longrightarrow&\mathcal T (X, \partial X)\nonumber\\
\downarrow&&\downarrow   \\
\mathcal S (\partial X)&\longrightarrow&\mathcal T (\partial X).\nonumber
\end{eqnarray}
which follows from
(\ref{2.13}).
\smallskip

For a compact  manifold with boundary
$(X^n, \partial X^n)$, we define a spectrum $\mathbb T\mathbb S(X, \partial X)$ as the homotopy cofiber
of  the  map
$\mathbb L(\pi_1(X))\longrightarrow \mathbb S(X, \partial X)$
which is a composition of maps from the diagram (\ref{4.5}) \cite{CavMurSpa06}.
Thus, as follows from the diagram (\ref{4.5}),  the spectrum $\Omega \mathbb T\mathbb S(X, \partial X)$
 fits into the following cofibrations
(cf. \cite{CavMurSpa06})
\begin{eqnarray}\label{4.10}
\Omega \mathbb T\mathbb S(X, \partial X) \longrightarrow& \mathbb L(\pi_1(X))&\longrightarrow \mathbb S(X, \partial X),\nonumber\\
\Omega \mathbb T\mathbb S(X, \partial X) \longrightarrow& (X/\partial X)_+\land\mathbf L_{\bullet}&\longrightarrow \Sigma \mathbb L(\pi_1(\partial X)),\\
\Omega \mathbb T\mathbb S(X, \partial X) \longrightarrow& \mathbb S(\partial X) &\longrightarrow \Sigma \left(X_+\land\mathbf L_{\bullet}\right).\nonumber
\end{eqnarray}

 Denote
$ \pi_i(\mathbb T\mathbb S(X, \partial X))=\mathcal T\mathcal S_i(X, \partial X)$,
and we have an isomorphism $ \mathcal T\mathcal S(X, \partial X)\cong \mathcal
T\mathcal S_{n+1}(X, \partial X)$.

The  mixed structure sets relate to other structure sets introduced above
by the following braids of exact sequences
\cite{CavMurSpa06}
\begin{equation}\label{4.11}
\xymatrix{\xyC{0.8pc}
 \ar[r] & \mathcal S_{n+1}(\partial X) \ar[rr] \ar[rd] &   & H_{n}(X;\mathbf L_{\bullet}) \ar[rr] \ar[rd]  &   & L_{n+1}( \pi_1(X) ) \ar[r] \ar[dr] &\\
\ar[ur] \ar[dr]   &   & \mathcal S_{n+1}^{\partial}(X, \partial X) \ar[dr] \ar[ur] & \ast & \mathcal T\mathcal S_{n+1}(X, \partial X) \ar[ur] \ar[dr] & &\\
 \ar[r]  & L_{n+1}( \pi_1(X) ) \ar[rr] \ar[ur]  &  & \mathcal S_{n+1}(X, \partial X) \ar[rr] \ar[ur] &  & \mathcal S_{n+1}(\partial X) \ar[r] \ar[ur] & }
\end{equation}
\begin{equation}\label{4.12}
\xymatrix{\xyC{0.8pc}
\ar[r] & H_{n+1}(X, \partial X;\mathbf L_{\bullet}) \ar[dr] \ar[rr]  &   & L_{n}(\pi_1(\partial X)) \ar[rr] \ar[dr] &  & \mathcal S_{n}(\partial X) \ar[r] \ar[dr] &\\
\ar[ur] \ar[dr] &   & H_n(\partial X; \mathbf L_{\bullet}) \ar[ur] \ar[dr] & & \mathcal T\mathcal S_{n+1}(X, \partial X) \ar[ur] \ar[dr] & \ast \ast & \\
\ar[r] & \mathcal S_{n+1}(\partial X) \ar[rr] \ar[ur] &  &  H_{n}(X; \mathbf L_{\bullet}) \ar[rr] \ar[ur] & \ar@{}[u]|-{\star}  & H_{n}(X, \partial X;\mathbf L_{\bullet}) \ar[ur] \ar[r] & }
\end{equation}
and
\begin{equation}\label{4.13}
\xymatrix{\xyC{0.8pc}
\ar[r] & L_{n+1}(\pi_1(X)) \ar[dr] \ar[rr]  &   & \mathcal S_{n+1}(X,\partial X)  \ar[rr] \ar[dr] & \ar@{}[d]|-{\ast} & H_{n}(X, \partial X;\mathbf L_{\bullet}) \ar[r] \ar[dr] &\\
\ar[ur] \ar[dr] &   & L_{n+1}^{rel} \ar[ur] \ar[dr] &  & \mathcal T\mathcal S_{n+1}(X, \partial X) \ar[ur] \ar[dr] &  & \\
\ar[r] & H_{n+1}(X, \partial X;\mathbf L_{\bullet} ) \ar[rr] \ar[ur] &  &  L_{n}(\pi_1(\partial X)) \ar[rr] \ar[ur] &   & {L}_{n}(\pi_1(X)) \ar[ur] \ar[r] & }
\end{equation}
which are realized on the spectrum level.

Note that the commutative square $*$  in (\ref{4.11}) is isomorphic to the left commutative square in 
the diagram (\ref{2.9}), the commutative square $**$  in (\ref{4.12}) is isomorphic to the right commutative
square in the diagram (\ref{2.13}), the commutative triangle $\star $ in (\ref{4.12}) follows from the right
square of (\ref{2.9}), and the commutative triangle $*$ in the diagram (\ref{4.13})  
is isomorphic to the commutative triangle
\begin{equation}\label{4.14}
\xymatrix{\xyC{0.8pc}
\mathcal S(X,\partial X) \ar[rr] \ar[dr] &  & \mathcal T(X,\partial X) \ar[dl]\\
& \mathcal T \mathcal S(X,\partial X) & \\
 }
\end{equation}
which follows from the diagram (\ref{2.13}).
\smallskip

Now let  $(X^n, Y^{n-q}, \xi)$  be a closed  manifold pair (cf. \cite{Ranicki81}*{\S 7.2}).
An $s$-triangulation  $f: M\to X$ {\sl splits} along  $Y$
if it is
homotopic to  a map  which is an $s$-triangulation of the manifold pair
$(X, Y, \xi)$ (the  definition was given  in Section 3).   The splitting obstruction group $LS_{n-q}(F)$
is  defined in  \cite{Ranicki81}*{\S 7.2}. Let  $\partial U$ be the boundary of a tubular
neighborhood $U$ of $Y$ in $X$.
The groups  $LS_{n-q}(F)$  depend  only on $n-q \bmod 4 $ and a pushout square
\begin{equation}\label{4.15}
F=\left(
\begin{array}{ccc}
 \pi_1(\partial U)&\to &\pi_1(X\setminus Y) \\
 \downarrow&&\downarrow\\
 \pi_1(Y)& \to&\pi_1(X) 
\end{array}
\right)
\end{equation}
of fundamental groups with  orientations.

If $f\colon M\to X$ is a $t$-triangulation then by \cite{Ranicki81}*{\S 7.2}
a  group $LP_{n-q}(F)$ of obstructions
is defined. This group  provides obstructions to the existence of an $s$-triangulation
of $(X, Y, \xi)$ in the class of the normal bordism of the map $f$.
The groups $LP_{n-q}(F)$
depend only on $n-q\bmod 4$  and the square $F$ as well.
There exists a homotopy commutative diagram
of spectra (cf. \cite{BakMur06} and \cite{Ranicki81})
\begin{eqnarray}\label{4.16}
\Omega\mathbb S^{\partial}(X\setminus Y, \partial(X\setminus Y))
\longrightarrow& (X\setminus Y)_+\land \mathbf L_{\bullet}&
\longrightarrow\mathbb L(\pi_1(X\setminus Y))\nonumber\\
\downarrow\qquad\quad&\downarrow&\qquad\quad\downarrow \nonumber\\
\Omega\mathbb S(X, Y, \xi)\longrightarrow&
  X_+\land \mathbf L_{\bullet}&
\longrightarrow
    \Sigma^q\mathbb L\mathbb P(F)  \\
\downarrow\qquad\quad&\downarrow&\qquad\quad\downarrow \nonumber\\
\Sigma^{q-1} \mathbb S(Y) \longrightarrow&\Sigma^{q} \left(Y_+\land \mathbf L_{\bullet}\right)&
 \longrightarrow\Sigma^{q} \mathbb L(\pi_1(Y)) \nonumber
\end{eqnarray}
where
the spectrum $\mathbb L\mathbb P(F)$ is the spectrum for obstruction groups $LP_i(F)$,
and all rows and columns in (\ref{4.16}) are cofibrations. Denote
$\pi_i(\mathbb S(X, Y, \xi))= \mathcal S_i(X,Y,\xi)$,  and we have an isomorphism \cite{Ranicki81}
\begin{equation}\label{4.17}
\mathcal S_{n+1}(X,Y,\xi)\cong \mathcal S(X,Y,\xi).
\end{equation}
The homotopy long exact sequences of rows and columns of the diagram (\ref{4.16}) give
the homotopy commutative diagram  of  groups
\begin{equation}\label{4.18}
\xymatrix{
 & \vdots \ar[d] & \vdots \ar[d] & \vdots \ar[d]   & &\\
\dots \ar[r] & \mathcal S_{n+1}^{\partial}(X\setminus Y, \partial(X\setminus
Y)) \ar[r] \ar@{}[dr]|-{\ast} \ar[d] & H_n(X\setminus Y;\mathbf L_{\bullet}) \ar[r] \ar[d] & L_n(\pi_1(X\setminus Y)) \ar[r] \ar[d] & \dots \\
\dots \ar[r] & \mathcal  S_{n+1}(X,Y, \xi) \ar@{}[dr]|-{\ast\ast} \ar[r] \ar[d] & H_n(X;\mathbf L_{\bullet})  \ar[r] \ar[d] & LP_{n-q}(F) \ar[r] \ar[d]& \dots \\
\dots \ar[r] & \mathcal  S_{n-q+1}(Y) \ar[r] \ar[d] & H_{n-q}(Y;\mathbf L_{\bullet}) \ar[d] \ar[r] & L_{n-q}(\pi_1(Y)) \ar[d] \ar[r] &\dots \\
  & \vdots  & \vdots  &\vdots  & }
\end{equation}
whose rows and columns are exact sequences.
This is the diagram  of \cite{Ranicki81}*{Proposition 7.2.6}.
Note that the commutative square $*$ in (\ref{4.18}) is isomorphic
to the commutative square
\begin{eqnarray}\label{4.19}
\mathcal S^{\partial} (X\setminus Y, \partial (X\setminus Y))&\longrightarrow&
\mathcal T^{\partial} (X\setminus Y, \partial (X\setminus Y))\nonumber\\
\downarrow\quad&&\quad\downarrow\\
\mathcal S (X,  Y, \xi)&\longrightarrow&\mathcal T (X)\nonumber
\end{eqnarray}
which follows from the
diagram  (\ref{3.1}). The commutative square $**$  in (\ref{4.18})
is isomorphic
to the commutative square
\begin{eqnarray}\label{4.20}
          \mathcal S(X, Y, \xi)&\longrightarrow&\mathcal T (X)\nonumber\\
\downarrow\quad&&\quad\downarrow\\
\mathcal S (Y)&\longrightarrow&\mathcal T (Y)\nonumber
\end{eqnarray}
which follows from the
diagram  (\ref{3.5}).
Thus the diagram (\ref{4.18}) is similar to the diagram (\ref{4.6}) which was constructed for a manifold
with boundary.

For a closed  manifold pair
$(X^n, Y^{n-q})$, we define a spectrum $\mathbb N\mathbb S(X, Y)$ as the homotopy fiber
of  the  map
$
\mathbb L(\pi_1(X\setminus Y))\longrightarrow \mathbb S(X, Y, \xi)
$
which is the composition of maps from the extended diagram (\ref{4.16}) \cite{BakMur06}.
It follows from (\ref{4.16}) (cf. \cite{BakMur06}) that this spectrum fits into the
following cofibrations
\begin{eqnarray}\label{4.21}
\mathbb N\mathbb S(X, Y)\longrightarrow&\mathbb L(\pi_1(X\setminus Y))&\longrightarrow\mathbb S(X, Y, \xi),\nonumber\\
\mathbb N\mathbb S(X, Y)\longrightarrow& X_+\land \mathbf L_{\bullet}&\longrightarrow\Sigma^q \mathbb L(\pi_1(Y)),\\
\mathbb N\mathbb S(X, Y)\longrightarrow& \Sigma^{q-1}\mathbb S(Y)&\longrightarrow
\Sigma\left((X\setminus Y)_+\land \mathbf L_{\bullet}\right).\nonumber
\end{eqnarray}
Note that the cofibrations in (\ref{4.21}) for $(X, Y)$ are similar to cofibrations in (\ref{4.10})
for $(X, \partial X)$.
Denote
$
\pi_i(\mathbb N\mathbb S(X, Y))=\mathcal N\mathcal S_i(X, Y),
$
and we have an isomorphism
$
\mathcal N\mathcal S(X, Y)\cong \mathcal N\mathcal S_{n}(X, Y)
$.

The  structure sets $\mathcal N\mathcal S(X,Y) $   relate to the other structure sets for a manifold pair
$(X,Y)$ by the following braids of exact sequences
\cite{BakMur06}
\begin{equation}\label{4.22}
\xymatrix{\xyC{0.1pc}
\ar[r] & \mathcal S_{n-q+2}( Y ) \ar[dr] \ar[rr]  &   & H_{n}(X\setminus Y; \mathbf L_{\bullet}) \ar[rr] \ar[dr] &  & L_{n}(\pi_1(X\setminus Y)) \ar[r] \ar[dr] &\\
\ar[ur] \ar[dr] &   & \mathcal  S_{n+1}^{\partial}(X\setminus Y, \partial(X\setminus Y)) \ar[ur] \ar[dr] & \ast & \mathcal N\mathcal S_{n}(X,Y) \ar[ur] \ar[dr] &  & \\
\ar[r] & L_{n+1}(\pi_1(X\setminus Y)) \ar[rr] \ar[ur] &  & \mathcal  S_{n+1}(X,Y,\xi)  \ar[rr] \ar[ur] &   & \mathcal  S_{n-q+1}(Y) \ar[ur] \ar[r] & }
\end{equation}
\begin{equation}\label{4.23}
\xymatrix{\xyC{0.7pc}
\ar[r] & H_{n+1}(X;\mathbf L_{\bullet}) \ar[dr] \ar[rr]  &   & L_{n-q+1}(\pi_1(Y)) \ar[rr] \ar[dr] &  & \mathcal  S_{n-q+1}(Y) \ar[r] \ar[dr] &\\
\ar[ur] \ar[dr] &   & H_{n-q+1}(Y; \mathbf L_{\bullet}) \ar[ur] \ar[dr] &  & \mathcal  N\mathcal  S_{n}(X,Y) \ar[ur] \ar[dr] & \ast \ast & \\
\ar[r] & \mathcal S_{n-q+2}(Y) \ar[rr] \ar[ur] &  & H_{n}(X\setminus Y;\mathbf L_{\bullet})  \ar[rr] \ar[ur] &   & H_{n}(X; \mathbf L) \ar[ur] \ar[r] & }
\end{equation}
\begin{equation}\label{4.24}
\xymatrix{\xyC{0.5pc}
\ar[r] & {L}_{n+1}(\pi_1(X\setminus Y)) \ar[dr] \ar[rr]  &   & \mathcal  S_{n+1}(X,Y, \xi ) \ar[rr] \ar[dr] &  & H_{n}(X; \mathbf L_{\bullet}) \ar[r] \ar[dr] &\\
\ar[ur] \ar[dr] &   & {LP}_{n-q+1}(F) \ar[ur] \ar[dr] &  & \mathcal  N \mathcal  S_{n}(X,Y) \ar[ur] \ar[dr] &  & \\
\ar[r] & H_{n+1}(X; \mathbf L_{\bullet}) \ar[rr] \ar[ur] &  & {L}_{n+1-q}(\pi_1(Y))  \ar[rr] \ar[ur] &   & {L}_{n}(\pi_1(X\setminus Y)) \ar[ur] \ar[r] & }
\end{equation}
The diagrams (\ref{4.22}) -- (\ref{4.24}) are realized on the spectrum level
and   follow
 from the diagrams (\ref{4.16}) (cf. \cite{BakMur03}, \cite{BakMur06},   and \cite{Muranov}).
The diagram (\ref{4.22}) is similar to the diagram (\ref{4.11}), the diagram (\ref{4.23})
is similar to the diagram (\ref{4.12}), and the diagram (\ref{4.24}) is similar to the
diagram (\ref{4.13}).
\smallskip

For the case of a manifold pair $(X,Y, \xi)$ we can consider the following
pair of topological spaces  $(X, X\setminus Y)$, where as usual
$X\setminus Y$ denotes the closure of $X\setminus U$ where
$U$ is a tubular neighborhood of $Y$ in $X$.
 Then the  commutative diagram
of exact sequences from  page 560 of \cite{Ranicki81} gives   the  diagram
\begin{equation}\label{4.25}
\xymatrix{
 & \vdots \ar[d] & \vdots \ar[d] & \vdots \ar[d]   & &\\
\dots \ar[r] & H_n(X\setminus Y;\mathbf L_{\bullet}) \ar[r] \ar[d] & L_{n}(\pi_1(X\setminus Y)) \ar[r] \ar[d] & \mathcal  S_{n}^{\partial}(X\setminus Y, \partial (X\setminus Y)) \ar[r] \ar[d] & \dots \\
\dots \ar[r] & H_n(X;\mathbf L_{\bullet}) \ar[r] \ar[d] & L_{n}(\pi_1(X)) \ar[r] \ar[d] & \mathcal  S_{n}(X) \ar[r] \ar[d]& \dots \\
\dots \ar[r] & H_{n}(X, X\setminus Y;\mathbf L_{\bullet}) \ar[r] \ar[d] & L_{n}^{rel} \ar[d] \ar[r] & \mathcal  S_{n}(X, X\setminus Y) \ar[d] \ar[r] &\dots \\
  & \vdots  & \vdots  &\vdots  & }
\end{equation}
where
$$
L_{n}^{rel}= L_n(\pi_1(X\setminus Y)\to \pi_1(X)), \quad H_{n}(X, X\setminus Y; \mathbf L_{\bullet}) \cong
H_{n-q}(Y; \mathbf L_{\bullet}).
$$

The structure sets for a manifold pair $(X, Y)$ also fit into  the following commutative diagrams
of exact sequences \cite{Ranicki81}*{\S  7.2}:
\begin{equation}\label{4.26}
\xymatrix{\xyC{0.8pc}
\ar[r] & L_{n+1}(\pi_1(X)) \ar[dr] \ar[rr]  &   & LS_{n-q}(F) \ar[rr] \ar[dr] &  & \mathcal S_{n}(X,Y, \xi) \ar[r] \ar[dr] &\\
\ar[ur] \ar[dr] &   & \mathcal S_{n+1}(X) \ar[ur] \ar[dr] &  & LP_{n-q}(F) \ar[ur] \ar[dr] &  & \\
\ar[r] & \mathcal S_{n+1}(X,Y, \xi ) \ar[rr] \ar[ur] & \ar@{}[u]|-{\ast} & H_{n}(X; \mathbf L_{\bullet})  \ar[rr] \ar[ur] &   & L_n(\pi_1(X)) \ar[ur] \ar[r], & }
\end{equation}
\begin{equation}\label{4.27}
\xymatrix{\xyC{0.2pc}
\ar[r] & \mathcal S_n^{\partial}(X\setminus Y, \partial(X\setminus Y) ) \ar[dr] \ar[rr]  & \ar@{}[d]|-{\ast}  & \mathcal S_{n}(X) \ar[rr] \ar[dr] &  & LS_{n-q-1}(F) \ar[r] \ar[dr] &\\
\ar[ur] \ar[dr] &   & \mathcal S_n(X, Y, \xi) \ar[ur] \ar[dr] &  & \mathcal S_n(X, X\setminus Y) \ar[ur] \ar[dr] &  & \\
\ar[r] & LS_{n-q}(F) \ar[rr] \ar[ur] &  & \mathcal S_{n-q}(Y)  \ar[rr] \ar[ur] &   & \mathcal S_{n-1}^{\partial}(X\setminus Y, \partial(X\setminus Y) ) \ar[ur] \ar[r] & }
\end{equation}
and
\begin{equation}\label{4.28}
\xymatrix{\xyC{0.8pc}
\ar[r] & H_{n-q}(Y; \mathbf L_{\bullet}) \ar[dr] \ar[rr]  &   & L_{n}^{rel} \ar[rr] \ar[dr] &  & LS_{n-q-1}(F) \ar[r] \ar[dr] &\\
\ar[ur] \ar[dr] &   & L_{n-q}(\pi_1(Y)) \ar[ur] \ar[dr] &  & \mathcal S_n(X, X\setminus Y) \ar[ur] \ar[dr] &  & \\
\ar[r] & LS_{n-q}(F) \ar[rr] \ar[ur] &  & \mathcal S_{n-q}(Y)  \ar[rr] \ar[ur] &   & H_{n-q}(Y; \mathbf L_{\bullet}) \ar[ur] \ar[r] & }
\end{equation}
which are realized on the spectrum level. The commutative triangle $*$ in (\ref{4.26})
corresponds to the following diagram of structures
\begin{equation}\label{4.29}
\xymatrix{\xyC{0.8pc}
\mathcal {SS} \ar[rr] \ar[dr] &  & \mathcal {TT} \\  
& \mathcal {ST} \ar[ur] & \\                          
 }
\end{equation}
which follows from the diagram (\ref{3.3}).

The commutative triangle $*$ in (\ref{4.27})
corresponds to the following diagram of structures
\begin{equation}\label{4.30}
\xymatrix{\xyC{0.8pc}
\mathcal {SD} \ar[rr] \ar[dr] &  & \mathcal {ST}\\
& \mathcal {SS}  \ar[ur] & \\
 }
\end{equation}
which follows from the diagram (\ref{3.3}).

For the case $(X, \partial X)$,  there are no  diagrams that are similar to the diagrams
(\ref{4.25})--(\ref{4.28}).  In this case we have a simple homotopy equivalence
$(X\setminus \partial X)\to X $,  and the theory is degenerated.

\section{The  structure sets of a manifold pair with boundary}

In this section we describe various  structure sets which
arise naturally for a compact  manifold pair with boundary in (\ref{1.1}).
 We shall use notations that are similar to notations of previous sections.
Thus, for example,
\begin{equation}\label{5.1}
\left[\begin{array}{cc}
\mathcal T&\mathcal T\\
\mathcal S&\mathcal D
\end{array}
\right]
\end{equation}
denotes a structure on $Q$ in (\ref{1.1}) which is a $\mathcal T\mathcal T$-structure on $(X,\partial X)$
 whose restriction to $(Y\hookleftarrow \partial Y)$ (the bottom pair of symbols) gives
an $\mathcal S \mathcal D$-structure,  and  whose restriction to $(\partial X\hookleftarrow \partial Y)$
(right vertical  pair of symbols) gives a $\mathcal T \mathcal D$-structure.
\smallskip

We  shall denote by $F$ the pushout square for the manifold pair
$(X, Y)$ as in (\ref{4.15}). The  manifold pair  with boundary $Q$
 in (\ref{1.1}) defines a pair of closed manifolds $\partial Y\subset \partial X$
with  a pushout square
\begin{equation*}
F_{\partial}=\left(
\begin{array}{ccc}
{\pi}_{1}(S(\partial\xi)) &
\longrightarrow & {\pi}_{1}(\partial X\setminus \partial Y) \\
\downarrow & \  & \downarrow \\
{\pi}_{1}(\partial Y) &
\longrightarrow & {\pi}_{1}(\partial X)
\end{array} \right)
\end{equation*}
of fundamental groups for the splitting problem. A natural inclusion
$\delta\colon \partial X \to X$ induces a map $\Delta\colon F_{\partial}\to
F$ of squares of fundamental groups. The relative splitting
obstruction groups $LS_*(\Delta)$ and the relative surgery obstruction
groups $LP_*(\Delta)$ are defined in \cite{CavMurSpa05}, \cite{CenMurRep}, \cite{Ranicki81},  and
\cite{Wall}.

First we describe structures for which there is a similarity
between the pairs  $(X, \partial X)$ and $(X, Y)$. In this case  for a  structure
$\left[\begin{array}{cc} \mathcal A \mathcal B\\
\mathcal C\mathcal E\\
\end{array}\right] $
we   have  that  $\mathcal B$ and $\mathcal C$ are  stronger than or equal to $\mathcal A$, and $\mathcal E$ is
stronger than or equal to $\mathcal B$ and $\mathcal C$.
In such a way it is sufficiently easy to describe all possible natural  structures on a manifold pair
with boundary in  (\ref{1.1}) (cf. \cite{CenMurRep} and \cite{Ranicki81}*{pp. 585-587}). 
In this case we obtain the following list of possible structures
\begin{equation}\label{5.2}
\left[\begin{array}{cc} \mathcal S \mathcal D\\
\mathcal D\mathcal D
\end{array}\right], \ \ \left[\begin{array}{cc} \mathcal S \mathcal S\\
\mathcal D\mathcal D
\end{array}\right], \
\ \left[\begin{array}{cc} \mathcal S\mathcal D\\
\mathcal S\mathcal D
\end{array}\right], \ \ \left[\begin{array}{cc} \mathcal S\mathcal S\\
\mathcal S\mathcal D\\
\end{array}\right], \ \ \left[\begin{array}{cc} \mathcal S\mathcal S\\
\mathcal S\mathcal S
\end{array}\right],
\end{equation}
\begin{equation}\label{5.3}
\left[\begin{array}{cc}\mathcal T \mathcal D\\
\mathcal D\mathcal D\\
\end{array}\right], \ \ \left[\begin{array}{cc} \mathcal T \mathcal T\\
\mathcal D\mathcal D\\
\end{array}\right], \
\ \left[\begin{array}{cc} \mathcal T\mathcal D\\
\mathcal T\mathcal D\\
\end{array}\right], \ \ \left[\begin{array}{cc} \mathcal T\mathcal T\\
\mathcal T\mathcal D\\
\end{array}\right], \ \ \left[\begin{array}{cc} \mathcal T\mathcal T\\
\mathcal T\mathcal T\\
\end{array}\right],
\end{equation}

\begin{equation}\label{5.4}
\left[\begin{array}{cc} \mathcal T \mathcal S\\
\mathcal S\mathcal S\\
\end{array}\right], \ \ \left[\begin{array}{cc} \mathcal T \mathcal T\\
\mathcal S\mathcal S\\
\end{array}\right], \
\ \left[\begin{array}{cc} \mathcal T\mathcal S\\
\mathcal T\mathcal S\\
\end{array}\right], \ \ \left[\begin{array}{cc} \mathcal T\mathcal T\\
\mathcal T\mathcal S\\
\end{array}\right],
\end{equation}
\begin{equation}\label{5.5}
\left[\begin{array}{cc} \mathcal T \mathcal S\\
\mathcal D\mathcal D\\
\end{array}\right], \ \ \left[\begin{array}{cc} \mathcal T \mathcal D\\
\mathcal S\mathcal D\\
\end{array}\right], \
\ \left[\begin{array}{cc} \mathcal T\mathcal S\\
\mathcal S\mathcal D\\
\end{array}\right], \ \ \left[\begin{array}{cc} \mathcal T\mathcal T\\
\mathcal S\mathcal D\\
\end{array}\right], \ \ \left[\begin{array}{cc} \mathcal T\mathcal S\\
\mathcal T\mathcal D\\
\end{array}\right].
\end{equation}

\smallskip
Note, that the structures in the lists (\ref{5.2}) -- (\ref{5.5}) are partially ordered
similarly to (\ref{2.2}) and (\ref{2.8}).  A structure
 \[
P_1=\left[\begin{array}{cc} \mathcal A_1\mathcal B_1\\
\mathcal C_1\mathcal E_1\\
\end{array}\right]
\]
is `stronger' than  a structure
 \[
P_2=\left[\begin{array}{cc} \mathcal A_2\mathcal B_2\\
\mathcal C_2\mathcal E_2\\
\end{array}\right]
\]
if $P_1\ne P_2$ and any  element fitting into  $P_1$ is `stronger than or equal' to
the corresponding element from $P_2$ in the sense of
(\ref{2.2}) and (\ref{2.8}). In this case  the natural map
$P_1\to P_2$ of `weakening of the structure' is defined. In particular this means,
that for any manifold $Q$ in (\ref{1.1}) we have a map
of structure sets
$
P_1(Q)\to P_2(Q).
$

We can also give the list of possible structures, which
correspond to the `exceptional'  structure  $\mathcal S\mathcal T$
on $(X, Y)$ or on $(\partial X, \partial Y)$, similarly to the end of
Section 3:
\begin{equation}\label{5.6}
\left[\begin{array}{cc} \mathcal S \mathcal D\\
\mathcal T\mathcal D\\
\end{array}\right], \ \ \left[\begin{array}{cc} \mathcal S \mathcal S\\
\mathcal T\mathcal D\\
\end{array}\right], \
\ \left[\begin{array}{cc} \mathcal S\mathcal S\\
\mathcal T\mathcal S\\
\end{array}\right], \ \ \left[\begin{array}{cc} \mathcal S\mathcal S\\
\mathcal T\mathcal T\\
\end{array}\right], \ \ \left[\begin{array}{cc} \mathcal T\mathcal S\\
\mathcal T\mathcal T\\
\end{array}\right].
\end{equation}

First we remark that several  structures from (\ref{5.2}) -- (\ref{5.6})
coincide with structures  which  were introduced  above  for the case $(X, \partial X)$.
It is easy to see that  these are  the following structures
\begin{equation}\label{5.7}
\left[\begin{array}{cc} \mathcal S \mathcal D\\
\mathcal D\mathcal D\\
\end{array}\right], \ \ \left[\begin{array}{cc} \mathcal T \mathcal D\\
\mathcal D\mathcal D\\
\end{array}\right], \
\ \left[\begin{array}{cc} \mathcal T\mathcal D\\
\mathcal T\mathcal D\\
\end{array}\right],
\ \ \left[\begin{array}{cc} \mathcal S\mathcal D\\
\mathcal T\mathcal D\\
\end{array}\right],
\ \ \left[\begin{array}{cc} \mathcal T\mathcal T\\
\mathcal T\mathcal T\\
\end{array}\right], \
 \ \left[\begin{array}{cc} \mathcal S\mathcal S\\
\mathcal T\mathcal T\\
\end{array}\right], \ \ \left[\begin{array}{cc} \mathcal T\mathcal S\\
\mathcal T\mathcal T\\
\end{array}\right].
\end{equation}

Note, that it is sufficiently easy  to give direct definition of other
structure sets  from (\ref{5.2}) -- (\ref{5.6}) similarly to Sections 2 and  3.
The naturality of the introduced structures (cf. \cite{BakMur06}, \cite{CavMurSpa06}, \cite{CenMurRep}, \cite{Ranicki79},
\cite{Ranicki81},  and \cite{Wall}) gives an opportunity
to construct
several
commutative diagrams of structures  which provide the
diagrams of structure sets, which are realized on the spectrum level --
similarly to the
diagrams (\ref{2.8}), (\ref{2.12}), (\ref{3.3}), and (\ref{3.5}).

Next we give several examples of such definitions and constructions
of commutative diagrams.

\begin{example}\label{Example 1} Here we consider structures which arise naturally
in the $rel_{\partial}$-case for (\ref{1.1}) (cf. \cite{CavMurSpa06}, \cite{CenMurRep}, \cite{Ranicki79},
\cite{Ranicki81},  and \cite{Wall}). These  are the structures on $Q$
in (\ref{1.1}) which have
the form $\left[\begin{array}{cc} \mathcal *\mathcal D\\
\mathcal *\mathcal D\\
\end{array}\right]$.

In particular, the first four  structures in (\ref{5.7}) are the structures of this type.
The first two of them  coincide with
 the structures $\mathcal S\mathcal D$ and $\mathcal T \mathcal D$, respectively,
  on the manifold with boundary
$(X\setminus Y, \partial (X\setminus Y))$. The third structure
is a $\mathcal T \mathcal D$-structure on $(X, \partial X)$. The fourth
structure  is an $\mathcal S\mathcal D$-structure on $(X, \partial X)$.

The   $\left[\begin{array}{cc} \mathcal S \mathcal D\\
\mathcal S\mathcal D\\
\end{array}\right]$-structure on a pair in (\ref{1.1})
corresponds to  $rel_{\partial}$-case structure set
\begin{equation}\label{5.8}
\mathcal S^{\partial}(X,Y,\xi)=\mathcal S^{\partial}(X,Y,\xi; \partial X)
\end{equation}
 defined in \cite{Ranicki81} (cf. also \cite{CenMurRep}).
In this case,  consider an $s$-triangulation
\begin{equation}\label{5.9}
(f,\partial f)\colon (M, \partial M)\to (X, \partial X)
\end{equation}
which is transversal to $(Y, \partial Y)$ with a transversal preimage
$(N, \partial N)$ such that $\partial f\colon \partial M\to \partial X$
is a homeomorfism and  the maps
\[
g=f_N\colon N\to Y,  \ \  \text{and} \ \ f|_{M\setminus N}=h \colon
 (M\setminus N, \partial (M\setminus N))\to (X\setminus Y, \partial (X\setminus Y))
 \]
 are $s$-triangulations (cf. \cite{BakMur06}, \cite{CenMurRep}, and \cite{Ranicki81}).
The set of concordance classes rel boundary
 of maps  in (\ref{5.9}) gives the structure set in (\ref{5.8}).

Using the same line of argument  as in Section 4  (cf. also \cite{BakMur06} and \cite{CenMurRep}),
we can define
$\left[\begin{array}{cc} \mathcal T\mathcal D\\
\mathcal S\mathcal D\\
\end{array}\right]$-structures on a pair in (\ref{1.1}), which
 corresponds to the structure set $\mathcal N\mathcal S^{\partial}(X,Y)$
($\mathcal N\mathcal S$-structure set relative boundary, as usual we suppose
that the restriction to the boundary is a homeomorphism).
\end{example}

\begin{definition} For a manifold pair $Q$  in (\ref{1.1}),
let
\[
(f, \partial f)\colon (M, \partial M) \to (X, \partial X)
\]
be a $t$-triangulation with a homeomorphism $\partial f$. Let $f$ be
transversal to $(Y, \partial Y)$ with  $(N, \partial N)=f^{-1}(Y, \partial Y)$, and  the restriction
$
g=f|_N\colon N\to Y
$
is a simple homotopy equivalence.
Two such maps
\[
f_i \colon M_i\to X,  \ \ N_i= f_i^{-1}(Y),  \ \ g_i=f_i|_{N_i}
\  (i=0,1)
\]
are {\sl  concordant}  if there
exists a $t$-triangulation
\[
(F;G,f_0, f_1)\colon
(W; M_0, M_1)\to
(X\times I;\partial X\times I,   X\times\{0\},
 X\times\{1\})
\]
 with  the following
properties:

\noindent (i) $\partial W = \partial M_0\times I\cup M_0\cup M_1$
and  $F|_{M_i}=f_i \ (i=0,1)$;

\noindent (ii) $G=g_0\times {\operatorname Id}\colon \partial M_0\times I\to \partial X\times I$;

\noindent (iii) $F$ is transversal to $Y\times I$  with $F^{-1}(Y)= V$ and $\partial V
=\partial N_0\times I\cup N_0\cup N_1$;

\noindent (iv) the restriction $F|_V$  is  an  $s$-triangulation of the 4-ad
\[
\left(F|_V; G|_{\partial N_0\times I},  f_0|_{N_0},f_1|_{N_1}\right)  \colon (V;N_0, N_1)\to (Y\times I; \partial Y\times I, Y\times\{0\},  Y\times\{1\})
\]
with
\[
G|_{\partial N_0\times I}=g_0|_{\partial N_0}\times {\operatorname Id}\colon \partial N_0\times I\to \partial Y\times I.
\]

The set of equivalence classes of such maps
is denoted by  $\mathcal N\mathcal S^{\partial}(X,Y)$ (cf.  \cite{BakMur06}).
\smallskip
\end{definition}

Similarly to Section 4 we can describe relations between structure
sets of Example \ref{Example 1}. The  structure set $\mathcal S^{\partial}(X,Y,\xi)$
is realized on the spectrum  level (cf. \cite{CenMurRep} and \cite{Ranicki81}) by a
spectrum $\mathbb S^{\partial}(X,Y, \xi)$ with
\[
\pi_i(\mathbb S^{\partial}(X,Y,\xi))=\mathcal S^{\partial}_i(X,Y,\xi) \ \ \text{and} \ \ \mathcal S^{\partial}_{n+1}(X,Y,\xi)=
\mathcal S^{\partial}(X,Y,\xi).
\]
This spectrum fits into the commutative diagram of cofibrations (cf. \cite{CenMurRep} and  \cite{Ranicki81})
\begin{eqnarray}\label{5.10}
\Omega\mathbb S^{\partial}(X\setminus Y, \partial(X\setminus Y))
\longrightarrow&(X\setminus Y)_+\land \mathbf L_{\bullet}&\longrightarrow
\mathbb L(\pi_1(X\setminus Y))\nonumber\\
\downarrow\qquad\qquad&\downarrow&\quad\qquad\downarrow\nonumber\\
\Omega\mathbb S^{\partial}(X, Y, \xi)\longrightarrow&
  X_+\land \mathbf L_{\bullet}&\longrightarrow
    \Sigma^q\mathbb L\mathbb P(F)  \\
\downarrow\qquad\qquad&\downarrow&\qquad\quad\downarrow\nonumber\\
\Sigma^{q-1} \mathbb S^{\partial}(Y)\longrightarrow&\Sigma^{q} \left(Y_+\land \mathbf L_{\bullet}\right)
&\longrightarrow\Sigma^{q} \mathbb L(\pi_1(Y))\nonumber
\end{eqnarray}
which is $rel_{\partial}$-version of the diagram (\ref{4.16}). The homotopy long exact sequences of (\ref{5.10})
provides the commutative diagram of exact sequences
\begin{equation}\label{5.11}
\xymatrix{
 & \vdots \ar[d] & \vdots \ar[d] & \vdots \ar[d]   & &\\
\dots \ar[r] & \mathcal S_{n+1}^{\partial}(X\setminus Y, \partial(X\setminus
Y)) \ar[r] \ar[d] & H_n(X\setminus Y;\mathbf L_{\bullet}) \ar[r] \ar[d] & L_n(\pi_1(X\setminus Y)) \ar[r] \ar[d] & \dots \\
\dots \ar[r] & \mathcal  S_{n+1}^{\partial}(X,Y, \xi) \ar[r] \ar[d] & H_n(X;\mathbf L_{\bullet}) \ar[r] \ar[d] & LP_{n-q}(F) \ar[r] \ar[d]& \dots \\
\dots \ar[r] & \mathcal  S_{n-q+1}^{\partial}(Y) \ar[r] \ar[d] & H_{n-q}(Y;\mathbf L_{\bullet}) \ar[d] \ar[r] & L_{n-q}(\pi_1(Y)) \ar[d] \ar[r] &\dots \\
  & \vdots  & \vdots  &\vdots  & }
\end{equation}
which is the $rel_{\partial}$-version of the diagram (\ref{4.18}).

Two left squares  in (\ref{5.11}) correspond to the following commutative
diagram of structures
\begin{eqnarray}
\left[\begin{array}{cc} \mathcal S \mathcal D\\
\mathcal D\mathcal D\\
\end{array}\right] &\longrightarrow&
 \left[\begin{array}{cc} \mathcal T \mathcal D\\
\mathcal D\mathcal D\\
\end{array}\right]\nonumber\\
\downarrow\quad&& \quad\downarrow\nonumber\\
\left[\begin{array}{cc} \mathcal S \mathcal D\\
\mathcal S\mathcal D\\
\end{array}\right] &\longrightarrow&
 \left[\begin{array}{cc} \mathcal T \mathcal D\\
\mathcal T\mathcal D\\
\end{array}\right]\\
\downarrow\quad&&\quad \downarrow\nonumber\\
\left[\begin{array}{cc} \mathcal S \mathcal D\\
\end{array}\right] &\longrightarrow&
 \left[\begin{array}{cc} \mathcal T \mathcal D\\
\end{array}\right] \nonumber
\end{eqnarray}
in which all arrows of the upper square and the bottom arrow correspond to the
`weakening of structures' and  the bottom vertical arrows  correspond to the
restrictions of the structures to the submanifold  $(Y, \partial
Y)$.
\smallskip

\begin{theorem} There exists a spectrum $\mathbb N\mathbb S^{\partial}(X, Y)$
with homotopy groups
\begin{equation}\label{5.13}
\pi_i(\mathbb N\mathbb S^{\partial}(X, Y))= \mathcal N \mathcal
S^{\partial}_i(X,Y), \ \
\mathcal N \mathcal S^{\partial}(X,Y)=\mathcal N \mathcal S^{\partial}_n(X,Y).
\end{equation}
\end{theorem}
\begin{proof} We can use the diagram (\ref{5.10}) similarly to the diagram
(\ref{4.16}) to define a spectrum $\mathbb N\mathbb S^{\partial}(X, Y)$
fitting into the cofibrations
\begin{eqnarray}\label{5.14}
\mathbb N\mathbb S^{\partial}(X, Y)\longrightarrow&\mathbb L(\pi_1(X\setminus Y))&\longrightarrow\mathbb S^{\partial}(X, Y, \xi)\nonumber\\
\mathbb N\mathbb S^{\partial}(X, Y)\longrightarrow &X_+\land \mathbf L_{\bullet}&\longrightarrow\Sigma^q \mathbb L(\pi_1(Y))\\
\mathbb N\mathbb S^{\partial}(X, Y)\longrightarrow &\Sigma^{q-1}\mathbb S^{\partial}(Y)&\longrightarrow \Sigma \left((X\setminus Y)_+\land \mathbf
L_{\bullet}\right).\nonumber
\end{eqnarray}
which are similar to the cofibrations in (\ref{4.21}).
From here the result follows using standard argument
(cf. \cite{BakMur06}, \cite{CenMurRep}, \cite{Ranicki81},  and \cite{Wall}).
\end{proof}

\begin{theorem} The groups  $\mathcal N \mathcal S^{\partial}_i(X,Y)$ fit into the following braids
of exact sequences
\begin{equation}\label{5.15}
\xymatrix{\xyC{0.1pc}
\ar[r] & \mathcal S_{n-q+2}^{\partial}( Y ) \ar[dr] \ar[rr]  &   & H_{n}(X\setminus Y; \mathbf L_{\bullet}) \ar[rr] \ar[dr] &  & L_{n}(\pi_1(X\setminus Y)) \ar[r] \ar[dr] &\\
\ar[ur] \ar[dr] &   & \mathcal  S_{n+1}^{\partial}(X\setminus Y, \partial(X\setminus Y)) \ar[ur] \ar[dr] & \circledast & \mathcal N\mathcal S_{n}^{\partial}(X,Y) \ar@{}[d]|-{\ast} \ar[ur] \ar[dr] &  & \\
\ar[r] & L_{n+1}(\pi_1(X\setminus Y)) \ar[rr] \ar[ur] &  & \mathcal  S_{n+1}^{\partial}(X,Y,\xi)  \ar[rr] \ar[ur] &   & \mathcal  S_{n-q+1}^{\partial}(Y) \ar[ur] \ar[r] & }
\end{equation}
\begin{equation}\label{5.16}
\xymatrix{\xyC{0.2pc}
\ar[r] & H_{n+1}(X;\mathbf L_{\bullet}) \ar[dr] \ar[rr]  &   & L_{n-q+1}(\pi_1(Y)) \ar[rr] \ar[dr] &  & \mathcal  S_{n-q+1}^{\partial}(Y) \ar[r] \ar[dr] &\\
\ar[ur] \ar[dr] &   & H_{n-q+1}(Y; \mathbf L_{\bullet}) \ar[ur] \ar[dr] &  & \mathcal  N\mathcal  S_{n}^{\partial}(X,Y) \ar@{}[d]|-{\ast \ast} \ar[ur] \ar[dr] & \circledast \circledast & \\
\ar[r] & \mathcal S_{n-q+2}^{\partial}(Y) \ar[rr] \ar[ur] &  & H_{n}(X\setminus Y;\mathbf L_{\bullet})  \ar[rr] \ar[ur] &   & H_{n}(X; \mathbf L) \ar[ur] \ar[r] & }
\end{equation}
\begin{equation}\label{5.17}
\xymatrix{\xyC{0.2pc}
\ar[r] & {L}_{n+1}(\pi_1(X\setminus Y)) \ar[dr] \ar[rr]  &   & \mathcal  S_{n+1}^{\partial}(X,Y, \xi ) \ar[rr] \ar[dr] & \ar@{}[d]|-{\ast \ast \ast} & H_{n}(X; \mathbf L_{\bullet}) \ar[r] \ar[dr] &\\
\ar[ur] \ar[dr] &   & {LP}_{n-q+1}(F) \ar[ur] \ar[dr] &  & \mathcal  N \mathcal  S_{n}^{\partial}(X,Y) \ar[ur] \ar[dr] &  & \\
\ar[r] & H_{n+1}(X; \mathbf L_{\bullet}) \ar[rr] \ar[ur] &  & {L}_{n+1-q}(\pi_1(Y))  \ar[rr] \ar[ur] &   & {L}_{n}(\pi_1(X\setminus Y)) \ar[ur] \ar[r] & }
\end{equation}
which are  realized on the spectrum level.
\end{theorem}
\begin{proof} The result follows from diagram (\ref{5.10})
similarly to the construction of the diagrams (\ref{4.22})--(\ref{4.24}) 
(cf. \cite{BakMur03}, \cite{BakMur06},   and \cite{Muranov}).
\end{proof}
Note that the diagrams (\ref{5.15}) -- (\ref{5.17}), in particular,  give several
natural relations
between structures from (\ref{5.2}) -- (\ref{5.7}) for (\ref{1.1}) and structures for $(Y, \partial Y)$
(cf. e.g.,  the  commutative triangles
$*$, $**$,   $***$, and the commutative squares  $\circledast$ and  $\circledast \circledast$).
\smallskip

\begin{example}\label{Example 2} To understand some of the structures from (\ref{5.2}) --
(\ref{5.6}) we need to use  structures on  manifold triads \cite{Wall}. Let
$\mathcal X=(X; \partial_0X,\partial_1 X)$ be a manifold triad \cite{Wall},
where
\[
\partial X=  \partial_0X\cup \partial_1 X \ \ \text{and let }  \partial_{\emptyset} X=  \partial_1X\cap \partial_0 X.
\]
We can write down this triad  in the following form
\begin{eqnarray}\label{5.17'}
\partial_{\emptyset} X &\longrightarrow& \partial_0X \nonumber\\
\downarrow& &\downarrow\\
\partial_{1} X &\longrightarrow& X \nonumber
\end{eqnarray}
where all the maps are induced by inclusions.
In this case, the results from
\S 10 of \cite{Wall} (cf. also \cite{Ranicki79} and \cite{Ranicki81}) provide the
following structure sets:

$\mathcal T(\mathcal X)$  --- the set of classes of normal bordisms of $t$-triangulations
of $\mathcal X$,

$\mathcal S(\mathcal X)$ ---  the set of classes of $s$-triangulations of $\mathcal X$,

$\mathcal T^{\partial_1X} (\mathcal X)$ --- the set of classes of normal bordisms of $t$-triangulations
of $\mathcal X$ rel $\partial_1X$,

$\mathcal S^{\partial_1X }(\mathcal X)$ --- the set of classes of $s$-triangulations of $\mathcal X$
rel $\partial_1X$,

and various structure sets of the pair $(\partial_1X,
\partial_{\emptyset}X)$. 

Let
\begin{equation}\label{5.18}
F=\left(\begin{array}{ccc}
\pi_1(\partial_{\emptyset} X) &\longrightarrow& \pi_1(\partial_0X) \\
\downarrow& &\downarrow\\
\pi_1(\partial_{1} X) &\longrightarrow& \pi_1(X )
\end{array}\right)
\end{equation}
be the square of fundamental groups with orientations. The structure
sets above   fit into the following commutative diagram of exact
sequences (cf.  \cite{Ranicki79}, \cite{Ranicki81}, and \cite{Wall})
 \begin{equation}\label{5.19}
\xymatrix{\xyC{0.8pc}\xyV{0.9pc}
&\ar[d]&     &\ar[d] &    & \ar[d] &       \\
\ar[r] &\mathcal S^{\partial_1X }_{n+1}(\mathcal X)
&\ar[r]&
\mathcal T^{\partial_1X }_{n}(\mathcal X)& \ar[r] & L_n(\pi_1(\partial_0X)\rightarrow \pi_1(X))\ar[r]& \\
&\ar[d]&   &\ar[d] &    & \ar[d]&    \\
\ar[r]&
\mathcal S_{n+1}(\mathcal X)  &\ar[r] &\mathcal T_{n}(\mathcal X)&
\ar[r] &L_n(F)\ar[r]& \\
&\ar[d]&   &\ar[d] &    & \ar[d]&   \\
\ar[r]& \mathcal S_{n}(\partial_1X, \partial_{\emptyset}X) \ar[d]
&\ar[r]&
\mathcal
T_{n-1}(\partial_1 X, \partial_{\emptyset}X)\ar[d]& \ar[r] &
L_{n-1}(\pi_1(\partial_{\emptyset}X)\rightarrow\pi_1(\partial_1 X))\ar[d]\ar[r]&\\
&&&&&}
\end{equation}
where $ \mathcal T_n(\mathcal X)\cong \mathcal T(\mathcal X)$, $\mathcal S_{n+1}(\mathcal
X)\cong \mathcal S(\mathcal X)$,
 $\mathcal T^{\partial_1X}_n (\mathcal X)\cong \mathcal
T^{\partial_1X} (\mathcal X)$, $\mathcal S^{\partial_1X }_{n+1}(\mathcal X)
\cong \mathcal S^{\partial_1X }(\mathcal X)$, $\mathcal S_{n}(\partial_1X,
\partial_{\emptyset}X)\cong \mathcal S(\partial_1X,
\partial_{\emptyset}X)$, and $\mathcal T_{n-1}(\partial_1 X, \partial_{\emptyset}X)
\cong \mathcal T(\partial_1 X, \partial_{\emptyset}X)$.
\end{example}

Wall pointed out the existence of mixed structures (cf. pages 115 and 116 in 
\cite{Wall}) of $n$-ads. Now, similarly to the definition of the $\mathcal T\mathcal
S$-structure on $(X, \partial X)$ in Section 2, we can define a
$\mathcal T\mathcal S$-structure on $ (\mathcal X;
\partial_1X)$ in (\ref{5.17'}).
This is the $\mathcal T$-structure on $\mathcal X$ whose restriction to
$(\partial_1X,
\partial_{\emptyset}X)$ is the $\mathcal S$-structure.

\begin{definition} Let  $f\colon  \mathcal  M \to  \mathcal X$ be a
$t$-triangulation
\begin{equation}\label{5.20}
\left(\begin{array}{ccc}
\partial_{\emptyset} M &\longrightarrow& \partial_0M \\
\downarrow& & \downarrow\\
\partial_{1} M &\longrightarrow& M \\
\end{array}\right)\longrightarrow
\left(\begin{array}{ccc}
\partial_{\emptyset} X &\longrightarrow& \partial_0X \\
\downarrow& &\downarrow\\
\partial_{1} X &\longrightarrow& X \\
\end{array}\right)
\end{equation}
of a triad in (\ref{5.17'}) such that the restriction
\[
f|_{(\partial_1 M, \partial_{\emptyset}M)} \colon (\partial_1 M,
\partial_{\emptyset}M) \to (\partial_1 X, \partial_{\emptyset}X)
\]
is an $s$-triangulation. Two such maps  $f_i \colon  \mathcal M_i \to
\mathcal X$, $i = 0, 1$,  are {\sl  concordant} if there
 exists a
$t$-triangulation  of the $4$-ad $\mathcal X\times I$ (cf. page 111 in 
\cite{Wall})
\[
F\colon (W; W_0\cup  W_1, V_0, V_1)\to (X\times I; X\times \{0\}\cup
X\times \{1\}, \partial_0X\times I, \partial_1X\times I)
\]
such that
\[
\partial V_i=\partial_iM_0\cup \partial_iM_1\cup S
\]
and
$
F|_{V_1}
$
 is  an $s$-triangulation
\[
(V_1;\partial_1 M_0,
\partial_1M_1, S)\to (\partial_1 X\times I;  \partial_1 X\times\{0\},
\partial_1  X\times\{1\}, \partial_{\emptyset}X\times I).
\]
The set of equivalence classes of such maps is denoted by $\mathcal T
\mathcal S(\mathcal X; \partial_1X)$.
\end{definition}
\smallskip

The diagram (\ref{5.19})  is realized on the spectrum level  by the
following homotopy commutative diagram of cofibrations
\begin{equation}\label{5.21}
\xymatrix{
 & \vdots \ar[d] & \vdots \ar[d] & \vdots \ar[d]   & &\\
\dots \ar[r] & \Omega\mathbb S^{\partial_1 X}(\mathcal X) \ar[r] \ar[d] & (X/\partial_0X)_+\land \mathbf L_{\bullet} \ar[r] \ar[d] & L(\pi_1(\partial_0 X)\to \pi_1(X)) \ar[r] \ar[d] & \dots \\
\dots \ar[r] & \Omega \mathbb S(\mathcal X) \ar[r] \ar[d] & (X/\partial X)_+\land \mathbf
L_{\bullet} \ar[r] \ar[d] & \mathbb L{F} \ar[r] \ar[d]& \dots \\
\dots \ar[r] & \mathbb S(\partial_1X, \partial_{\emptyset} X) \ar[r] \ar[d] & \Sigma
(\partial_1 X/\partial_{\emptyset} X)_+\land \mathbf L_{\bullet} \ar[d] \ar[r] & \Sigma \mathbb L(\pi_1(\partial_{\emptyset} X)\to
\pi_1(\partial_1X)) \ar[d] \ar[r] &\dots \\
  & \vdots  & \vdots  &\vdots  & }
\end{equation}
which is similar to the diagram (\ref{4.5}).

\begin{theorem} There exists a spectrum $\mathbb T\mathbb S(\mathcal X;
\partial_1X)$ with homotopy groups
$
\pi_i( \mathbb T\mathbb S(\mathcal X;
\partial_1X))=\mathcal T\mathcal S_{i}(\mathcal X; \partial_1X)
$
fitting into the cofibrations
\begin{equation}\label{5.22}
\xymatrix{
 \Omega \mathbb T\mathbb S (\mathcal X;\partial_1X)&\ar[r]&\mathbb L(\pi_1(\partial_0X)\to \pi_1(X))&\ar[r]&\mathbb S(\mathcal X),\\
\Omega \mathbb T\mathbb S(\mathcal X;\partial_1X)&\ar[r]&(X/\partial X)_+\land\mathbf L_{\bullet}&\ar[r]&\Sigma \mathbb L(\pi_1(\partial_{\emptyset} X)\to \pi_1(\partial_1 X)),\\
\Omega \mathbb T\mathbb S(\mathcal X;\partial_1X) &\ar[r]&\mathbb S(\partial_1 X, \partial_{\emptyset}X)&\ar[r]&\Sigma \left((X/\partial_0X)_+\land\mathbf L_{\bullet}\right)\\
}
\end{equation}
such that
$
\mathcal T \mathcal S(\mathcal X; \partial_1X)=\mathcal T \mathcal S_n(\mathcal X;
\partial_1X).
$
\end{theorem}
\begin{proof} We can use the diagram (\ref{5.21}) similarly to the diagram (\ref{4.5})
to define a spectrum $\mathbb T\mathbb S(\mathcal X;
\partial_1X)$ fitting into
the cofibrations (\ref{5.22}).
 From this the result
follows using standard line of  argument  (cf. \cite{CavMurSpa06} and
\cite{Wall}).
 \end{proof}
 
 \begin{remark} Similarly to the construction of the diagrams
 (\ref{4.11})--(\ref{4.13}), it is now easy  to write down braids of exact
 sequences which connect the groups $\mathcal T \mathcal S_i(\mathcal X;
 \partial_1X)$ with others groups from the diagram (\ref{5.19}).
 \end{remark}

\begin{example}\label{Example 3} Several structures from (\ref{5.2}) -- (\ref{5.6}) coincide
with the structures  on the corresponding manifold triads. For example,
for a pair $Q$ in (\ref{1.1}) consider the manifold triad
\begin{equation}\label{5.23}
\mathcal Z=(X\setminus Y; \partial_0(X\setminus Y),\partial_1
(X\setminus Y))= (X\setminus Y;
\partial X\setminus \partial Y, \partial U)
\end{equation}
 where $\partial U$ is a tubular
neighborhood of $Y$ in $X$, $X\setminus Y= \overline{X\setminus U}$,
and similarly the definition we use  for $\partial X\setminus
\partial Y$. Now,  the second structure in (\ref{5.2}) and the second
structure in (\ref{5.3}) provide the structure sets which coincide with
$\mathcal S^{\partial_1(X\setminus Y)}(X\setminus Y;
\partial X\setminus \partial Y, \partial U)$ and $\mathcal T^{\partial_1(X\setminus Y)}(X\setminus Y;
\partial X\setminus \partial Y, \partial U,)$ of the triple (\ref{5.23}), respectively.

For $Q$ in (\ref{1.1}),  consider the  manifold triad
\begin{equation}\label{5.24}
\mathcal X=(X; \partial_0X,\partial_1 X)=(X;
\partial X\setminus \partial Y, U_{\partial Y})
\end{equation}
where $U_{\partial Y}$ is a tubular neighborhood of $\partial Y$ in
$\partial X$. The second structure in (\ref{5.6}) is $\mathcal S^{\partial_1
X}(\mathcal X)$-structure on $\mathcal X$ in (\ref{5.24}), and the third structure
in (\ref{5.6}) is the structure $\mathcal S(\mathcal X)$ on $\mathcal X$ in (\ref{5.24})
(cf.  \cite{Ranicki79}, \cite{Ranicki81}, \cite{Wall}). The fourth structure in (\ref{5.3}) is $\mathcal
T^{\partial_1X}(\mathcal X)$-structure on $\mathcal X$ in (\ref{5.24}).
\end{example}

\begin{example}\label{Example 4} In the conditions of Example \ref{Example 2} for a triple in
(\ref{5.17}) we can write down the following commutative diagram (cf.
\cite{Ranicki81} and \cite{Wall})
\begin{equation}\label{5.24'}
\xymatrix{\xyC{0.3pc}\xyV{0.9pc}
&\ar[d]&     &\ar[d] &    & \ar[d] &       \\
\ar[r] &\mathcal S^{\partial  }_{n+1}( X, \partial X)
&\ar[r]&
H_n(X; \mathbf L_{\bullet})& \ar[r] & L_n(\pi_1(X))\ar[r]& \\
&\ar[d]&   &\ar[d] &    & \ar[d]&    \\
\ar[r]&
\mathcal S_{n+1}^{\partial_1 X}(\mathcal X) &\ar[r] &H_n(X,
\partial_0X; \mathbf L_{\bullet})&
\ar[r] &L_n(\pi_1(\partial_0X)\to \pi_1(X))\ar[r]& \\
&\ar[d]&   &\ar[d] &    & \ar[d]&   \\
\ar[r]& \mathcal S_{n}^{\partial}(\partial_0X, \partial_{\emptyset}X)\ar[d]
&\ar[r]&
H_{n-1}(\partial_0X; \mathbf L_{\bullet})\ar[d]& \ar[r] &
L_{n-1}(\pi_1(\partial_0 X))\ar[d]\ar[r]&\\
&&&&&}
\end{equation}
where $H_n(X; \mathbf L_{\bullet})= \mathcal T^{\partial  }_{n}( X,
\partial X)$,
$ H_n(X,
\partial_0X; \mathbf L_{\bullet})=\mathcal T^{\partial_1X}_n (\mathcal X)$, and $H_{n-1}(\partial_0X; \mathbf L_{\bullet})=
\mathcal T_{n-1}^{\partial}(\partial_0 X,\partial_{\emptyset}X)$.

The diagram (\ref{5.24'}) describes connections between various structure
sets and obstruction groups for the case of a manifold triple $Q$ in
(\ref{5.17}) rel $\partial_1X$. Similarly to the consideration above we
can introduce the  mixed structure set $\mathcal T \mathcal S^{\partial_1X}(\mathcal X,
\partial_0X)$ on $Q$ rel $\partial_1X$. This structure
set consists of the classes of concordance of $t$-triangulations of
$\mathcal X$ rel $\partial_1X$ whose restrictions to
$\partial_0X$ provide classes of concordance of  $s$-triangulations
rel $\partial_{\emptyset}X$. Since the diagram (\ref{5.24'}) is realized
on the spectrum level the introduced structure set  is realized on
the spectrum level by the spectrum $\mathbb T \mathbb S^{\partial_1X}(\mathcal
X,\partial_0X)$ with
\[
\pi_i(\mathbb T \mathbb S^{\partial_1X}(\mathcal X,\partial_0X))=\mathcal T \mathcal
S^{\partial_1X}_i(\mathcal X,
\partial_0X), \ \
\mathcal T \mathcal S^{\partial_1X}_n(\mathcal X,
\partial_0X)\cong \mathcal T \mathcal
S^{\partial_1X}(\mathcal X,\partial_0X).
\]
\end{example} 

Since the diagram (\ref{4.24}) is realized on the spectrum level, similarly to
Theorem 3,  we obtain the following result.

\begin{theorem} The spectrum $\mathbb T \mathbb S^{\partial_1X}(\mathcal
X,\partial_0X)$ fits  into the following cofibrations
\begin{eqnarray}\label{5.25}
\Omega \mathbb T \mathbb S^{\partial_1X}(\mathcal
X,\partial_0X)\longrightarrow &\mathbb L(\pi_1(X))&\longrightarrow \mathbb S^{\partial_1X}(\mathcal X)\nonumber\\
\Omega \mathbb T \mathbb S^{\partial_1X}(\mathcal
X,\partial_0X) \longrightarrow &(X/\partial_0 X)_+\land\mathbf L_{\bullet}&\longrightarrow \Sigma \mathbb L(\pi_1(\partial_{0} X))\\
\Omega \mathbb T \mathbb S^{\partial_1X}(\mathcal
X,\partial_0X) \longrightarrow &\mathbb S^{\partial}(\partial_0 X, \partial_{\emptyset}X)&\longrightarrow \Sigma \left(X_+\land\mathbf L_{\bullet}\right).\nonumber
\end{eqnarray}
\end{theorem}
\hfill $\Box$

\begin{remark} Similarly to the construction of the diagrams (\ref{4.11})--(\ref{4.13}), it
is now easy to write down braids of exact sequences which connect
the groups $\mathcal T \mathcal S^{\partial_1X}_i(\mathcal X,
\partial_0X)$ with other groups from the diagram (\ref{5.24'}).
\end{remark}
\smallskip

\begin{example}\label{Example 5} The first structure in (\ref{5.5}) is presented by the
structure set  $\mathcal T \mathcal S^{\partial_1 Z}(\mathcal Z,
\partial_0Z)$ from (\ref{5.23}) and the last structure in (\ref{5.5}) is presented by
the structure set $\mathcal T \mathcal S^{\partial_1X}(\mathcal X,
\partial_0X)$ from (\ref{5.24}).
\end{example}
 
\begin{example}\label{Example 6} The  structure $\left[\begin{array}{cc} \mathcal S \mathcal S\\
\mathcal S\mathcal S\\
\end{array}\right]$  in (\ref{5.2}) is given by the classes
of concordance of $s$-triangulations of $(X, \partial X)$ which are
split along $Y$ and split along $\partial Y$ on the boundary
(cf. \cite{CenMurRep} and \cite{Ranicki81}).  This structure corresponds to the
structure set
\[
\mathcal S(X,Y; \partial)=\mathcal S(X,Y;\partial X, \partial Y)=\mathcal S(X,
Y, \xi;  \partial X,
\partial Y, \partial \xi)
\]
described
 in  \cite{CenMurRep}.
\end{example}

\begin{example}\label{Example 7} The third structure in (\ref{5.4}) is given by the
classes of concordance of $t$-triangulations of $(X,
\partial X)$ whose restrictions to the boundary provides  classes of concordance
of $s$-triangulations $\partial X$ which are split along
$\partial Y$. Denote the corresponding structure set by $ \mathcal T\mathcal
S(X; \partial X, \partial Y) $.
 This structure set is realized by the spectrum $
\mathbb T\mathbb  S(X; \partial X, \partial Y) $ with
\[
\pi_i(\mathbb T\mathbb  S(X; \partial X, \partial Y)) = \mathcal T\mathcal S_i(X;
\partial X, \partial Y),\ \ \mathcal T\mathcal S_n(X; \partial X, \partial
Y)=\mathcal T\mathcal S(X; \partial X, \partial Y)
\]
(cf. \cite{CenMurRep} and \cite{Ranicki81}).
\end{example} 

Consider the commutative diagram of structures on $Q$ in (\ref{1.1})
\begin{equation}\label{5.26}
\begin{array}{ccc}
\left[\begin{array}{cc} \mathcal S \mathcal S\\
\mathcal D\mathcal D\\
\end{array}\right]& \longrightarrow&
 \left[\begin{array}{cc} \mathcal T \mathcal T\\
\mathcal D\mathcal D\\
\end{array}\right]\\
\downarrow\quad& &\quad\downarrow\\
\left[\begin{array}{cc} \mathcal S \mathcal S\\
\mathcal S\mathcal S\\
\end{array}\right] &\longrightarrow&
 \left[\begin{array}{cc} \mathcal T \mathcal T\\
\mathcal T\mathcal T\\
\end{array}\right]\\
\downarrow\quad& &\quad\downarrow\\
\left[\begin{array}{cc} \mathcal S \mathcal S\\
\end{array}\right] &\longrightarrow&
 \left[\begin{array}{cc} \mathcal T \mathcal T\\
\end{array}\right]
\end{array}
\end{equation}
where all the maps in the upper square and the bottom horizontal map are the
maps of `weakening of the structure', and the vertical maps in the bottom
square are given by the restriction of the structure to the submanifold.
The diagram (\ref{5.26}) induces the following commutative diagram of
structure sets (cf. \cite{CenMurRep}, \cite{Ranicki79}, \cite{Ranicki81}, and \cite{Wall})
\begin{equation}\label{5.27}
\xymatrix{
 & \vdots \ar[d] & \vdots \ar[d] & \vdots \ar[d]   & &\\
\dots \ar[r] & \mathcal S_{n+1}^{\partial U}(X\setminus Y; \partial
X\setminus \partial Y, \partial U) \ar@{}[dr]|-{\ast} \ar[r] \ar[d] & H_n(X\setminus Y, \partial X\setminus \partial Y;\mathbf L_{\bullet})  \ar[r] \ar[d] & L_n^{rel\ \partial  U} \ar[r] \ar[d] & \dots \\
\dots \ar[r] & \mathcal  S_{n+1}(X,Y; \partial X, \partial Y) \ar@{}[dr]|-{\ast\ast} \ar[r] \ar[d] &  H_n(X,
\partial X;\mathbf L_{\bullet}) \ar[r] \ar[d] & LP_{n-q}(\triangle) \ar[r] \ar[d]& \dots \\
\dots \ar[r] &  \mathcal  S_{n-q+1}(Y, \partial Y) \ar[r] \ar[d] & H_{n-q}(Y, \partial
Y;\mathbf L_{\bullet}) \ar[d] \ar[r] & L_{n-q}^{rel} \ar[d] \ar[r] &\dots \\
  & \vdots  & \vdots  &\vdots  & }
\end{equation}
where the first row is the surgery exact sequence of the triple in
(\ref{5.23}) for the case rel $\partial U$ with
$
L_*^{rel\
\partial U}=L_*(\pi_1(\partial X\setminus
\partial Y)\to \pi_1(X\setminus Y))
$ and $ L_*^{rel}= L_{*}(\pi_1(\partial Y)\to \pi_1(Y))$. The second
row of (\ref{5.27}) follows from the diagram 3.34 of \cite{CenMurRep}.
 The bottom row in (\ref{5.27}) is the surgery exact
sequence of the manifold pair $(Y, \partial Y)$.

The diagram (\ref{5.27}) is
realized on the spectrum level and we have the following
isomorphisms
\begin{eqnarray*}
 \mathcal S^{\partial U}(X\setminus Y; \partial X\setminus \partial
Y,\partial U)&\cong& \mathcal S_{n+1}^{\partial U}(X\setminus Y; \partial
X\setminus
\partial Y, \partial U),
\\
 \mathcal S(X,Y;\partial X, \partial Y)&\cong& \mathcal  S_{n+1}(X,Y;
\partial X, \partial Y),\\
 \mathcal T^{\partial U}(X\setminus Y; \partial X\setminus \partial Y,
\partial U)&\cong&
 H_n(X\setminus Y, \partial X\setminus \partial Y;\mathbf
 L_{\bullet}).
 \end{eqnarray*}

\begin{example}\label{Example 8} Using the same line of argument  as above  we can
define a structure set  $\mathcal N\mathcal S^{rel}(X, Y;  \partial X,
\partial Y)$ which consists of the classes of $\mathcal N\mathcal S$-structures
which give  $\mathcal N\mathcal S$-structures on the boundary $(\partial X,
\partial Y)$.
 This case corresponds to
 the $\left[\begin{array}{cc} \mathcal T\mathcal T\\
\mathcal S\mathcal S\\
\end{array}\right]$-structure from the list in (\ref{5.4}) on the manifold $Q$ in (\ref{1.1}).
\end{example}

\begin{theorem} There exists a spectrum $\mathbb N\mathbb S^{rel}(X,
Y;
\partial X, \partial Y)$ with homotopy groups
\[
\pi_i(\mathbb N\mathbb S^{rel}(X, Y; \partial X, \partial Y))= \mathcal N \mathcal S^{rel}_i(X,Y; \partial X, \partial Y), \ \
\mathcal N \mathcal S^{rel}(X,Y; \partial X, \partial Y)=\mathcal N \mathcal
S^{rel}_n(X,Y; \partial X, \partial Y).
\]
This spectrum fits into the cofibrations
\begin{eqnarray*}
 \mathbb N\mathbb S^{rel}(X, Y; \partial X, \partial Y)\longrightarrow &\mathbb
L(\pi_1(\partial X\setminus \partial Y)\to \pi_1(X\setminus Y))&\longrightarrow
\mathbb S^{}(X, Y;  \partial),\\
 \mathbb N\mathbb S^{rel}(X, Y; \partial X, \partial Y)\longrightarrow &(X/\partial
X)_+\land \mathbf L_{\bullet}&\longrightarrow \Sigma^q \mathbb L(\pi_1(\partial
Y)\to\pi_1(Y)),\\
\mathbb N\mathbb S^{rel}(X, Y; \partial X, \partial Y)\longrightarrow&\Sigma^{q-1}\mathbb
S(Y, \partial Y)&\longrightarrow \Sigma^1 \left(\left[(X\setminus Y)/(\partial
X\setminus Y)\right]_+\land \mathbf L_{\bullet}\right).
\end{eqnarray*}
\end{theorem}
\begin{proof} The standard line of argument  as before provides this
result (cf. also \cite{BakMur06}, \cite{CenMurRep}, \cite{Ranicki81},  and \cite{Wall}). 
\end{proof}


\begin{remark} Similarly to the constructions above  it is now easy  to
write down braids of exact sequences which connect the groups $\mathcal
N \mathcal S^{rel}_i(X,Y;
\partial X, \partial Y)$ with other groups from the diagram (\ref{5.27}).
\end{remark}

\begin{remark} Note that using the standard argument  of this section
it is possible to define other structure sets for the structures
\begin{equation}\label{5.31}
\left[\begin{array}{cc} \mathcal S \mathcal S\\
\mathcal S\mathcal D\\
\end{array}\right], \ \ \left[\begin{array}{cc} \mathcal T \mathcal S\\
\mathcal S\mathcal S\\
\end{array}\right], \
\ \left[\begin{array}{cc} \mathcal T\mathcal S\\
\mathcal S\mathcal D\\
\end{array}\right], \ \ \left[\begin{array}{cc} \mathcal T\mathcal T\\
\mathcal T\mathcal S\\
\end{array}\right], \ \ \left[\begin{array}{cc} \mathcal T\mathcal T\\
\mathcal S\mathcal D\\
\end{array}\right]
\end{equation}
from (\ref{5.2})--(\ref{5.5}) which are  not defined above.
\end{remark}

\section{On the surgery obstruction groups}

The constructions of spectra for various structure sets provides
also the natural maps of spectra which correspond to the maps of
`weakening of the structure' for structure sets. The cofibres of such
maps are the surgery spectra for the corresponding obstruction groups.
Below we consider several examples.

\begin{example}\label{Example 9} For a closed manifold $X$ we have only one map
(\ref{2.2}) of `weakening of the structure'. This map induces the cofibration
\begin{equation}\label{6.1}
\Omega \mathbb S(X) \to X_+\land \mathbf L_{\bullet}\rightarrow \mathbb
L(\pi_1(X))
\end{equation}
which is equivalent to the cofibration (\ref{4.1}).
 \end{example}

\begin{example}\label{Example 10} For a manifold with boundary $(X, \partial X)$ we
have the diagram (\ref{2.8}) for structures.   Thus we obtain the following
 maps of `weakening of the structure' and corresponding
cofibrations:
\begin{equation}\label{6.2}
\begin{array}{c}
 \mathcal S\mathcal D\to \mathcal T\mathcal D,\\
 \Omega \mathbb S^{\partial}(X, \partial X) \to X_+\land \mathbf L_{\bullet}\rightarrow \mathbb
L(\pi_1(X))
\end{array}
\end{equation}
which coincide with the upper row in (\ref{4.5});
\begin{equation}\label{6.3}
\begin{array}{c}
 \mathcal S\mathcal S\to \mathcal T\mathcal T,\\
 \Omega \mathbb S(X, \partial X)\to  (X/\partial X)_+\land \mathbf L_{\bullet}
 \to
\mathbb L^{rel} \end{array}
\end{equation}
which coincide with the middle  row in (\ref{4.5});
\begin{equation}\label{6.4}
\begin{array}{c}
 \mathcal T\mathcal S\to \mathcal T\mathcal T,\\
 \Omega \mathbb T\mathbb S(X, \partial X) \to (X/\partial X)_+\land\mathbf L_{\bullet}\to \Sigma \mathbb L(\pi_1(\partial X))
 \end{array}
\end{equation}
which coincide with the middle  row in (\ref{4.10});
\begin{equation}\label{6.5}
\begin{array}{c}
 \mathcal T\mathcal D\to \mathcal T\mathcal T,\\
 X_+\land \mathbf L_{\bullet}\to  (X/\partial X)_+\land \mathbf L_{\bullet}
 \to
\Sigma\left((\partial X_+)\land \mathbf L_{\bullet}\right)
\end{array}
\end{equation}
which coincide with the middle  column  in (\ref{4.5});
\begin{equation}\label{6.6}
\begin{array}{c}
 \mathcal S\mathcal D\to \mathcal S\mathcal S,\\
 \mathbb S^{\partial}(X, \partial X)\to  \mathbb S(X, \partial X)\to \Sigma \mathbb S(\partial X)
\end{array}
\end{equation}
which coincide with the left column  in (\ref{4.5});
\begin{equation}\label{6.7}
\begin{array}{c}
 \mathcal T\mathcal D\to \mathcal T\mathcal S,\\
 X_+\land \mathbf L_{\bullet}\to \Omega \mathbb T\mathbb S(X, \partial X) \to \mathbb S(\partial X)
 \end{array}
\end{equation}
which is equivalent to the bottom row in (\ref{4.10});
\begin{equation}\label{6.8}
\begin{array}{c}
 \mathcal S\mathcal S\to \mathcal T\mathcal S,\\
\mathbb S(X, \partial X)\to  \mathbb T\mathbb S(X, \partial X)  \to \Sigma
\mathbb L(\pi_1(X)) \end{array}
\end{equation}
which is equivalent to the upper row in (\ref{4.10}). For the maps
\begin{equation}\label{6.9}
\begin{array}{c}
 \mathcal S\mathcal D\to \mathcal T\mathcal S,\\
 \mathcal S\mathcal D\to \mathcal T\mathcal T\\
\end{array}
\end{equation}
which follows from (\ref{2.8}) the obstruction groups 
have not been introduced so far. Now it is an easy exercise to introduce the surgery
spectra for  the corresponding obstruction groups and commutative
diagrams and braids of exact sequences, similarly to above.
\end{example}

For the case of a manifold pair $(X,Y)$ we can write down the maps
of `weakening of the structure' and corresponding cofibrations
similarly to Example \ref{Example 10}. The diagram (\ref{3.1}) on the spectrum level
provides the maps of spectra for structure sets, and as cofibres of
the maps we shall obtain spectra for corresponding obstruction
groups.  In this case we can also write additional cofibrations  of
spectra for structure sets (with cofibres which are spectra for
corresponding obstruction groups)  which provides commutative
diagram of structures (\ref{3.3}).

For a manifold pair with boundary we have a great number of  maps
of `weakening of the structure', but now it is sufficiently easy to
construct spectra for corresponding obstruction groups and describe
theirs properties. Many of such examples are given in Section 5.

\begin{example}\label{Example 11} For a manifold pair with boundary (\ref{1.1}) there is
the following map of `weakening of the structure'
\begin{equation}\label{6.10}
\left[\begin{array}{cc} \mathcal S \mathcal S\\
\mathcal S\mathcal S\\
\end{array}\right] \longrightarrow \left[\begin{array}{cc} \mathcal T \mathcal S\\
\mathcal T\mathcal T\\
\end{array}\right]
\end{equation}
which provides the following cofibration (cf. Theorem 2 of \cite{CenMurRep})
\begin{equation}\label{6.11}
 \mathbb S(X, Y, \partial )\to \mathbb T\mathbb S(X,\partial X)\longrightarrow
\Sigma^{q+1}\mathbb L\mathbb P\mathbb S(\Delta)
\end{equation}
where the groups $ LPS_i(\Delta)=\pi_i(\mathbb L\mathbb P\mathbb S(\Delta))$
are the corresponding obstruction groups.
\end{example}

\section*{Acknowledgements}
This research was supported by the
 Slovenian Research Agency grants P1-0292-0101 and J1-9643-0101.

\bigskip
\end{document}